\documentclass[11pt,a4paper]{article}

\title{Calibration for multivariate L\'evy-driven Ornstein-Uhlenbeck processes with applications to weak subordination}
\author{
    Kevin W.~Lu\thanks{Department of Applied Mathematics,
        University of Washington,
        Seattle, WA 98195,
        United States.
        Email: \href{mailto:kwlu@uw.edu}{kwlu@uw.edu}}
}

\usepackage{amsmath,amsfonts,amssymb,amsthm}
\usepackage{graphicx}

\usepackage[shortlabels]{enumitem} 
\usepackage[colorlinks,citecolor=red,linkcolor=blue]{hyperref} 
\usepackage{fancyhdr} 
\usepackage{microtype} 
\usepackage{datetime2} 

\numberwithin{equation}{section} 
\allowdisplaybreaks[4] 
\let\originalleft\left
\let\originalright\right
\def\left#1{\mathopen{}\originalleft#1}
\def\right#1{\originalright#1\mathclose{}}

\newtheorem{theorem}{Theorem}[section]
\newtheorem{corollary}[theorem]{Corollary}
\newtheorem{lemma}[theorem]{Lemma}
\newtheorem{proposition}[theorem]{Proposition}

\theoremstyle{definition}
\newtheorem{definition}[theorem]{Definition}

\theoremstyle{remark}
\newtheorem{remark}[theorem]{Remark}

\renewcommand{\today}
{\number\day \space \ifcase\month\or
    January\or February\or March\or April\or May\or June\or
    July\or August\or September\or October\or November\or December
    \fi \space \number\year
}

\newcommand{\rmd}{{\rm d}}
\newcommand{\rmi}{{\rm i}}
\newcommand{\stp}{\stackrel{P}{\rightarrow}}
\newcommand{\std}{\stackrel{D}{\rightarrow}}
\newcommand{\stas}{\stackrel{\text{a.s.}}{\longrightarrow}}
\newcommand{\eqd}{\stackrel{D}{=}}

\newcommand{\DD}{\mathbb{D}}
\newcommand{\EE}{\mathbb{E}}

\newcommand{\RR}{\mathbb{R}}
\newcommand{\NN}{\mathbb{N}}

\newcommand{\PP}{\mathbb{P}}

\newcommand{\PPP}{{\cal P}}

\newcommand{\GGG}{{\cal G}}

\newcommand{\LLL}{{\cal L}}

\newcommand{\TTT}{{\cal T}}

\newcommand{\XXX}{{\cal X}}

\newcommand{\ZZZ}{{\cal Z}}

\newcommand{\skal}[2]{\langle #1,#2\rangle}

\newcommand{\eins}{{\bf 1}}

\newcommand{\bfnull}{{\bf 0}}

\newcommand{\bfe}{{\bf e}}

\newcommand{\bft}{{\bf t}}
\newcommand{\bfx}{{\bf x}}
\newcommand{\bfy}{{\bf y}}

\newcommand{\bfz}{{\bf z}}
\newcommand{\bfd}{{\bf d}}

\newcommand{\bfB}{{\bf B}}
\newcommand{\bfC}{{\bf C}}
\newcommand{\bfJ}{{\bf J}}

\newcommand{\bfT}{{\bf T}}
\newcommand{\bfU}{{\bf U}}
\newcommand{\bfV}{{\bf V}}
\newcommand{\bfW}{{\bf W}}
\newcommand{\bfX}{{\bf X}}
\newcommand{\bfY}{{\bf Y}}
\newcommand{\bfZ}{{\bf Z}}
\newcommand{\bfalpha}{\boldsymbol{\alpha}}
\newcommand{\bfdelta}{\boldsymbol{\delta}}

\newcommand{\bfmu}{\boldsymbol{\mu}}

\newcommand{\bftheta}{\boldsymbol{\theta}}
\newcommand{\bfvtheta}{\boldsymbol{\vartheta}}
\newcommand{\bfeta}{\boldsymbol{\eta}}
\newcommand{\bfzeta}{\boldsymbol{\zeta}}

\newcommand{\myCov}{\operatorname{Cov}}

\newcommand{\myCorr}{\operatorname{Corr}}

\newcommand{\spur}{\operatorname{trace}}

\newcommand{\wt}{\widetilde}
\newcommand{\tr}{\diamond}

\newcommand{\given}{{\,\vert\,}}
\newcommand{\givenm}{\,\middle\vert\,}

\begin{document}
\date{\today}
\maketitle

\begin{abstract}
    Consider a multivariate L\'evy-driven Ornstein-Uhlenbeck process where the stationary distribution or background driving L\'evy process is from a parametric family. We derive the likelihood function assuming that the innovation term is absolutely continuous. Two examples are studied in detail: the process where the stationary distribution or background driving L\'evy process is given by a weak variance alpha-gamma process, which is a multivariate generalisation of the variance gamma process created using weak subordination. In the former case, we give an explicit representation of the background driving L\'evy process, leading to an innovation term which is discrete and continuous mixture, allowing for the exact simulation of the process, and a separate likelihood function. In the latter case, we show the innovation term is absolutely continuous. The results of a simulation study demonstrate that maximum likelihood numerically computed using Fourier inversion can be applied to accurately estimate the parameters in both cases.

    \vspace{0.5em} \noindent {\em Keywords:} L\'evy process, Ornstein-Uhlenbeck process, self-decompos\-ability, likelihood inference, multivariate subordination, weak subordination, variance gamma process.
       
    \vspace{0.5em} \noindent {\em 2010 MSC Subject Classification:} Primary: 62M05, 60G51, 60G10; Secondary: 62F10, 60E10, 60H05.  
\end{abstract}

\section{Introduction}\label{intro}

Let $\bfX=(\bfX(t))_{t\geq 0}$ be the $n$-dimensional process given by the stochastic differential equation
\begin{align}\label{ldoup}
    \rmd\bfX(t) =-\lambda \bfX(t)\rmd t + \rmd \bfZ(\lambda t),\quad \bfX(0)=\bfX_0,\quad t\geq 0,
\end{align}
where $\lambda> 0$, $\bfZ$ is an $n$-dimensional L\'evy process, and $\bfX_0$ is a random vector independent of $\bfZ$. Here, $\bfX$ is known as a \emph{L\'evy-driven Ornstein-Uhlenbeck process (LDOUP)} or a Ornstein-Uhlenbeck-type process, $\lambda$ is the autocorrelation parameter, and $\bfZ$ is the \emph{background driving L\'evy process (BDLP)}.

Suppose that the stationary distribution or the BDLP of the LDOUP is from a parametric family with parameter vector $\wt{\bfvtheta}$, then the law of $\bfX$ is determined by the parameter vector ${\bfvtheta}:=(\lambda,\wt\bfvtheta)$. Let $t_0=0,t_1=\Delta , \dots,t_m= m\Delta$ be $m+1$ equally spaced observation times with sampling interval $\Delta>0$. In this paper, we consider the estimation of the parameter vector $\bfvtheta$ based on the observations $\bfX(0),\bfX(t_1),\dots, \bfX(t_m)$ using maximum likelihood (ML) with Fourier inversion of a characteristic function, as well as the simulation of the observations. In many respects, this work can be seen as an multivariate generalisation of Valdivieso, Schoutens and Tuerlinckx \cite{vstt}.

The classical Ornstein-Uhlenbeck process, that is \eqref{ldoup} with the BDLP being a univariate Brownian motion, is well-known and widely used. In the Vasicek model, short-term interest rates are modelled by a classical Ornstein-Uhlenbeck process plus a long-run mean. The generalisation to L\'evy-driven Ornstein-Uhlenbeck processes was introduced by Sato and Yamazato \cite{saya,SaYa85}, and there have been various applications to mathematical finance. Specifically, in the {Barndorff-Nielsen and Shephard model} \cite{BNSh01} (also see \cite{BN97,sco}), stochastic volatility for stock prices and exchange rates are modelled by a LDOUP where the BDLP is a subordinator, that is a nondecreasing L\'evy process.
Typically, the BDLP is chosen such that the stationary distribution is gamma or inverse Gaussian. Cariboni and Schoutens \cite{CS092} have used the same processes to model {default intensities} in credit risk.

{Mean-reverting price processes} with jumps are often modelled using LDOUPs. Accordingly, these processes have been applied to the modelling of energy prices and the pricing of energy derivatives \cite{BKMB,CKM18,Sab20}. As an example, in \cite{BKMB}, electricity spot prices follow a weighted sum of univariate LDOUPs representing price factors, plus a deterministic function representing trend and seasonality. In Endres and St\"{u}binger \cite{ES19} and Wu, Zang and Zhao \cite{WZZ20}, optimal pair trading strategies are studied under a model where price spreads follow a LDOUP plus a deterministic function, and in a similar price spread model, Benth and Saltyte-Benth \cite{BSB06} consider the pricing of spark spread options.


Parameter estimation for the classical Ornstein-Uhlenbeck process is well-established, see for instance \cite{AS02,LiSh}. Parameter estimation for univariate LDOUPs in our setting has been studied in Valdivieso, Schoutens and Tuerlinckx \cite{vstt} using maximum likelihood, and in Taufer and Leonenko \cite{TL092} using empirical characteristic functions, while in the context of stochastic volatility models, estimation has been studied in \cite{BNSh01,GS06,PFH14}. Asymptotic results using maximum likelihood are considered in Zhang, Zhang and Shuguang \cite{ZhZhSu06} for the case where the stationary distribution is gamma, and using empirical characteristic functions in Jongbloed and van der Meulen \cite{JvdM06} for the case where the BLDP is a driftless subordinator. Numerous authors \cite{BDY07,GPR20,HL09,Mai14,Tra17,ZhZh13} have also worked on asymptotic results for estimators of the autocorrelation parameter $\lambda$ of a univariate LDOUP in a range of settings, sometimes along with other parameters, however they do not deal with the situation of the BDLP coming from a more general parametric family. In Jongbloed, van der Meulen and van der Vaart \cite{JMV05}, the authors consider nonparametric estimation of the stationary distribution for univariate LDOUPs via its L\'evy density. Fasen \cite{Fas13} studies the estimation of the autocorrelation matrix for multivariate LDOUPs. Overall, there has been relatively fewer research on LDOUPs in the multivariate setting.



The calibration method in this paper is based on maximum likelihood estimation. The discrete observations of a LDOUP follow an AR(1) process, $\bfX(t_k) = b\bfX(t_{k-1}) + \bfZ_{b}^{(k)}$, $k=1,\dots,m$, where $b=e^{-\lambda\Delta}$ and $\bfZ_{b}^{(k)}\eqd e^{-\lambda \Delta} \bfZ^*(\Delta)$ is iid. The innovation term $\bfZ^*(\Delta):=\int_0^{\lambda\Delta}e^{s}\rmd \bfZ(s)$ plays a critical role, occurring ubiquitously in our analysis. This AR(1) structure and the properties of $\bfZ^*(\Delta)$ are used to derive likelihood functions and simulation methods.


If we now let $\bfX=(X_1,\dots,X_n)$ be a L\'evy process and $\bfT=(T_1,\dots,T_n)$ be a subordinator, representing time change, and assume $\bfX$ and $\bfT$ are independent, then the process $\bfX\circ\bfT = (X_1(T_1(t)),\dots,X_n(T_n(t)))_{t\geq 0}$ is known as the {strong subordination} of $\bfX$ and $\bfT$. If $\bfT$ has indistinguishable components or $\bfX$ has independent components, then $\bfX \circ \bfT$ is a L\'evy process (see \cite[Theorem 30.1]{sato99} for the first case and \cite[Theorem 3.3]{BPS01} for the second case), otherwise it may not be (see \cite[Proposition 3.9]{BLM17a}). As a result, L\'evy processes created by strong subordination have a restrictive dependence structure. This motivated the introduction of weak subordination in Buchmann, Lu and Madan \cite{BLM17a}, a method for constructing time-changed L\'evy processes which extends and reproduces properties similar to strong subordination. In the variance alpha-gamma process introduced by Semeraro \cite{LS10,Se08}, a multivariate Brownian motion is subordinated by the sum of a univariate gamma subordinator affecting all components of the process to represent a common time change, and univariate gamma subordinators independently affecting each component of the process to represent idiosyncratic time changes. The result is a L\'evy process provided the Brownian motion has independent components. In contrast, the weak variance alpha-gamma (WVAG) process introduced in \cite{BLM17a} is constructed similarly, but using weak subordination instead, and now allowing the Brownian motion to have possibly dependent components while the result remains a L\'evy process. These WVAG processes have been applied to model multivariate stock returns and in pricing of best-of options \cite{BLM17c,MiSz17}.

Here, we focus specifically on two examples using subordination, the WVAG-OU and OU-WVAG processes. The former is a LDOUP where the stationary distribution is a WVAG distribution and the latter is a LDOUP where the BDLP is a WVAG process. In all, the WVAG-OU and OU-WVAG processes are models where the stationary distribution or BDLP, respectively, can be interpreted as a multivariate time-changed process with jumps, heavier tails and a flexible dependence structure, while also exhibiting mean reversion and exponentially decaying autocorrelation from the Ornstein-Uhlenbeck structure. For the WVAG-OU process, we give an explicit representation of the BDLP in terms of a compound Poisson distribution with the base distribution being a mixture of variance gamma distributions, so the LDOUP has finite activity. The corresponding innovation term $\bfZ^*(\Delta)$ is then a discrete and continuous mixture, its characteristic function is known in closed-form, and we give an explicit representation that is used to exactly simulate the LDOUP and derive the likelihood function for the process. There is similar work on the exact simulation of LDOUPs with various other stationary distributions (for example, gamma \cite{QDZ19}, inverse Gaussian \cite{ZhZh08}, generalised inverse Gaussian \cite{Zh11}, tempered stable \cite{ZhZh09}, bilateral gamma \cite{SCP21}, and multivariate tempered stable \cite{Grab20}). For the OU-WVAG process, we show that $\bfZ^*(\Delta)$ is absolutely continuous. This process has infinite activity. However, there is no known closed-form formula for the characteristic exponent of $\bfZ^*(\Delta)$, and the methods for the WVAG-OU process are not applicable in this case.

Next, we extend the likelihood and simulation methods of Valdivieso, Schoutens and Tuerlinckx \cite{vstt} to the multivariate setting. While these methods can be used for the OU-WVAG process, they also apply to general LDOUPs. Specifically, assuming $\bfZ^*(\Delta)$ is absolutely continuous, we give a likelihood function in terms of the Lebesgue density of $\bfZ^*(\Delta)$, which in turn can be computed using Fourier inversion on the characteristic exponent of $\bfZ^*(\Delta)$. The stochastic integral $\bfZ^*(\Delta)$ can be approximated with an Euler scheme, and provided that it is possible to simulate from the BDLP $\bfZ$, we can simulate the observations of $\bfX$. The convergence of such approximations is studied in more generality in \cite{JP98}, though we also provide a simple and direct proof specialised to our situation. We also provide moment formulas for $\bfZ^*(\Delta)$.

Then we perform a simulation study to show that the parameters for both the WVAG-OU and OU-WVAG processes can be accurately estimated, though we use a sequential method to approximate maximum likelihood in the latter case given the high computational burden from not having a closed-form formula. This leads to less accurate results compared to the WVAG-OU process.


To summarise, the main contributions of this paper are, firstly, we find an explicit representation of $\bfZ$ and $\bfZ^*(\Delta)$ in Theorem \ref{wvagouchar} for the WVAG-OU process, and use it to derive a likelihood function despite the absolute continuity assumption failing and give an exact simulation method. Secondly, we give multivariate versions of the likelihood and simulation methods in Valdivieso, Schoutens and Tuerlinckx \cite{vstt} and show this is applicable to the OU-WVAG process. And lastly, we numerically implement these calibration and simulation methods for both processes.

The paper is structured in the following way. In Section \ref{sec2}, we give a brief outline of L\'evy processes, LDOUPs and weak subordination. In Section \ref{sec4}, we focus on the specific example of the WVAG-OU and OU-WVAG processes and obtain the above-mentioned results. In Section \ref{sec3}, we give the likelihood function for a LDOUP assuming $\bfZ^*(\Delta)$ is absolutely continuous, and discuss the approximate simulation method. In Section \ref{sec5}, we simulate WVAG-OU and OU-WVAG processes, provide the calibration results from a Monte Carlo simulation study, and conclude with a discussion. All proofs are contained in Section \ref{sec6}. Appendix \ref{append} reviews the connection between LDOUPs and self-decomposability.

\section{Preliminaries}\label{sec2}

We write $\bfx=(x_1,\dots,x_n)\in\RR^n$ as a row vector. For $A\subseteq\RR^n$, let $A_*:=A\backslash\{{\bf 0}\}$ and let $\eins_A$ denote the indicator function for $A$. Let $\DD:=\{\bfx\in\RR^n:\|\bfx\|\le 1\}$ be the Euclidean unit ball centred at the origin. For $\bfx\in\RR^n$ and $\Sigma\in\RR^{n\times n}$, let $\|\bfx\|^2_\Sigma:=\bfx\Sigma\bfx'$. Let $\bfe_k$, $k=1,\dots,n$, be the $k$th canonical basis vector of $\RR^n$. Let $I:[0,\infty)\to[0,\infty)$ be the identity function. Let $\Phi_{\bfX}$, $\Psi_{\bfX}$, $f_{\bfX}$ and $\PPP_{\bfX}$ respectively denote the characteristic function, characteristic exponent, density, and probability law of the random vector $\bfX$, or of $\bfX(1)$ if $\bfX$ is a process. Let $\LLL^n$ be the $n$-dimensional Lebesgue measure with $\LLL := \LLL^1$, $\bfdelta_{\bfx}$ be the Dirac measure at $\bfx\in\RR^n$, and $\otimes$ denote the product measure. We write a process $\bfX$ in terms of its marginal components and time marginals as $\bfX=(X_1,\dots,X_n)=(\bfX(t))_{t\ge 0}$.

\subsection{L\'evy Processes}

For references on L\'evy processes, see \cite{bert96,sato99}. The law of an $n$-dimensional L\'evy process $\bfX$ is determined by its characteristic function $\Phi_\bfX:=\Phi_{\bfX(1)}$, where
\begin{align*}
    \Phi_{\bfX(t)}(\bftheta)\,:=\,\EE[e^{\rmi\skal\bftheta{\bfX(t)}}] = e^{t\Psi_\bfX(\bftheta)}, \quad \text{$t\ge0$, $\bftheta\in\RR^n$},
\end{align*}
or its characteristic exponent
\begin{align}\label{lkform}
    \Psi_\bfX(\bftheta)=
    \rmi \skal {\bfmu}\bftheta - \frac 12\|\bftheta\|^2_{\Sigma} +\int_{\RR^n_*}\left(e^{\rmi\skal\bftheta \bfx} - 1 - \rmi\skal\bftheta \bfx \eins_\DD(\bfx)\right)\,\XXX(\rmd \bfx),
\end{align}
where $\bfmu\in\RR^n$, $\Sigma\in\RR^{n\times n}$ is a covariance matrix  and $\XXX$ is a L\'evy measure, that is a nonnegative Borel measure on $\RR^n_*$ such that
$\int_{\RR^n_*}(1\wedge \|\bfx\|^2)\,\XXX(\rmd \bfx)<\infty$. We write $\bfX\sim L^n(\bfmu,\Sigma,\XXX)$ to mean $\bfX$ is an $n$-dimensional L\'evy process with characteristic triplet $(\bfmu,\Sigma,\XXX)$.


An $n$-dimensional L\'evy process $\bfT$ with almost surely nondecreasing sample paths is called a \emph{subordinator}, and denoted $\bfT \sim S^n(\bfd,\TTT)$, where $\bfd := \bfmu - \int_{\DD_*} \bft \,\TTT(\rmd \bft)$ is its drift, and $\TTT$ is its L\'evy measure.

Let $P_S(b)$ denote a Poisson process with rate $b>0$. Let $CP^n(b,\PPP)$ denote an $n$-dimen\-sional compound Poisson process with rate $b>0$ and base distribution $\PPP$ on $\RR^n_*$.
Whenever we append $\bfeta\in\RR^n$ to the end of the list of parameters of a L\'evy process, that means we add the drift $\bfeta I$ to that process, so for example, $\bfC\sim  CP^n(b,\PPP,\bfeta)$ means $\bfC= \bfeta I + \wt\bfC$, where $\wt\bfC\sim CP^n(b,\PPP)$.

%
%

\subsection{Weak Subordination}

We give a brief outline of weak subordination. These weakly subordinated processes will be used as the stationary distribution or BDLP. The results in this subsection can be found in \cite{BLM17a}.



Let $\bfX=(X_1,\dots,X_n)\sim L^n$ and $\bfT =(T_1,\dots,T_n)\sim S^n(\bfd,\TTT)$. For all $\bft=(t_1,\dots,t_n)\in[0,\infty)^n$, the random vector $\bfX(\bft):= (X_1(t_1),\dots X_n(t_n))$ is infinitely divisible and its characteristic exponent is denoted $\bft \tr \Psi_{\bfX}$. The \emph{weak subordination} of $\bfX$ and $\bfT$ is the L\'evy process $\bfZ\eqd\bfX\odot\bfT$ defined by the characteristic exponent
\begin{align*}
    \Psi_\bfZ(\bftheta) =(\bfd\tr\Psi_{\bfX})(\bftheta)+\int_{[0,\infty)^n_*} (\Phi_{\bfX(\bft)}(\bftheta)-1)\,\TTT(\rmd \bft), \quad \bftheta\in\RR^n.
\end{align*}
The strong subordination $\bfX\circ\bfT$ is not always a L\'evy process, but the weak subordination $\bfX\odot\bfT$ is, and in the cases mentioned in Section~\ref{intro} where it is known that $\bfX\circ\bfT$ is a L\'evy process,  $\bfX\circ\bfT\eqd\bfX\odot\bfT$. Weak subordination has other properties and jump behavior that are similar to strong subordination.

Let $\Gamma_S(a,b)$ denote a gamma subordinator with shape $a>0$ and rate $b>0$, and $BM^n(\bfmu,\Sigma)$ denote a Brownian motion with drift $\bfmu=(\mu_1,\dots,\mu_n)\in\RR^n$ and covariance matrix $\Sigma=(\Sigma_{kl})\in\RR^{n\times n}$.

Let $b>0$. A process $\bfV\sim VG^n(b,\bfmu,\Sigma)$ is a \emph{variance gamma (VG) process} if $\bfV\eqd \bfB\circ (G,\dots,G)$, where $\bfB\sim BM^n(\bfmu,\Sigma)$ and $G\sim\Gamma_S(b,b)$ are independent.

Let $n\geq 2$, $a> 0$, $\bfalpha=(\alpha_1,\dots,\alpha_n) \in(0,1/a)^n$ and $\beta_k:= (1-a\alpha_k)/{\alpha_k}$, $k=1,\dots,n$. Let $G_{0}\sim\Gamma_S(a,1)$, $G_k\sim\Gamma_S(\beta_k,1/\alpha_k)$, $k=1,\dots,n$, be independent. A process $\wt\bfW \sim WVAG^n(a,\bfalpha,\allowbreak\bfmu,\Sigma)$ is a \emph{weak variance alpha-gamma (WVAG) process} if $\wt\bfW\eqd \bfB\odot \bfT$, where $\bfB\sim BM^n(\bfmu,\Sigma)$ and $\bfT= G_0\bfalpha+(G_1,\dots,G_n)$. This is a multivariate generalisation of the VG process with a flexible dependence structure without requiring that the Brownian motion have independent components.
Define $\bfalpha\tr\bfmu:= (\alpha_1\mu_1,\allowbreak\dots,\alpha_n\mu_n)\in\RR^n$ and $\bfalpha\tr\Sigma:=(\Sigma_{kl} (\alpha_k\allowbreak\wedge \alpha_l))\in\RR^{n\times n}$,
then the characteristic exponent of $\bfW\sim WVAG^n(a,\bfalpha,\bfmu,\Sigma,\bfeta)$ is
\begin{align}
    \begin{split}
        \Psi_\bfW(\bftheta)={}&\rmi\skal{\bfeta}{\bftheta}-a\log\left( 1-\rmi \skal{\bfalpha\tr\bfmu}\bftheta+\frac{1}{2}\|\bftheta\|^2_{\bfalpha\tr\Sigma}\right) \\
        &-\sum_{k=1}^n\beta_k\log\left(1-\rmi \alpha_k\mu_k\theta_k+\frac{1}{2}\alpha_k\theta_k^2\Sigma_{kk}\right),\quad \bftheta\in\RR^n.\label{v-alpha-g-cf}
    \end{split}
\end{align}



\subsection{L\'evy-Driven Ornstein-Uhlenbeck Processes}

%

Recalling the definition of a LDOUP in \eqref{ldoup}, note that there is no loss in generality in using the L\'evy process $\bfZ\circ (\lambda I)$ instead of $\bfZ$ as the driving noise since any L\'evy process $\wt\bfZ\sim L^n(\wt\bfmu,\wt\Sigma,\wt\ZZZ)$ can be written in the form $\wt\bfZ\eqd\bfZ\circ (\lambda I)$, where $\bfZ\sim L^n(\wt\bfmu/\lambda,\wt\Sigma/\lambda,\wt\ZZZ/\lambda)$. 

A LDOUP is the solution to \eqref{ldoup}, which is
\begin{align}\label{ousoln}
    \bfX(t)=e^{-\lambda t} \bfX(0) +e^{-\lambda t}\int_0^t e^{\lambda s}\,\rmd\bfZ(\lambda s),\quad t\geq 0,
\end{align}
where the integral term is a stochastic integral with respect to a L\'evy process (see \cite{mas04,saya}).

We are interested in the situation where the LDOUP $\bfX$ is stationary, that is there exists a random vector $\bfY$ such that $\bfX_0\eqd \bfY$ implies $\bfX(t)\eqd \bfY$ for all $t\geq 0$. Then the distribution of $\bfY$ (or $\bfY$, with a minor abuse of terminology) is known as the \emph{stationary distribution} of $\bfX$.
An $n$-dimensional random vector $\bfY\sim SD^n$ is called \emph{self-decomposable} if for any $0 < b < 1$, there exists a random vector $\bfZ_b$, independent of $\bfY$, such that
\begin{align}\label{sddefn}
    \bfY\eqd b\bfY+\bfZ_b.
\end{align}
A L\'evy process $\bfY$ is self-decomposable if $\bfY(1)$ is. See Appendix \ref{append} for a summary of the connection between LDOUPs and self-decomposability. Specifically, due to Lemma \ref{statsolnlem}, a stationary LDOUP can be defined in two equivalent ways: either by specifying its BDLP or its stationary distribution.
\begin{enumerate}[(M1)]
    \item  Let $\bfX\sim \operatorname{\mathit{OU-\bfZ}}(\lambda)$ be the stationary LDOUP with autocorrelation parameter $\lambda>0$ and BDLP $\bfZ\sim L^n(\bfmu,\Sigma,\ZZZ)$ satisfying
    \begin{align}\label{logcond}
        \int_{(2\DD)^C}\log \|\bfz\|\,\ZZZ(\rmd\bfz)<\infty.
    \end{align}
    There exists a corresponding stationary distribution $\bfY\sim SD^n$.\label{model1}
    \item  Let $\bfX\sim \operatorname{\mathit{\bfY-OU}}(\lambda)$ be the stationary LDOUP with autocorrelation parameter $\lambda>0$ and stationary distribution $\bfY\sim SD^n$. There exists a corresponding BDLP $\bfZ\sim L^n(\bfmu,\Sigma,\ZZZ)$ satisfying \eqref{logcond}. \label{model2}
\end{enumerate}

This naming convention follows \cite[Section 5.2.2]{sco}. Since the LDOUP $\bfX$ is assumed to have a stationary distribution, the initial distribution $\bfX_0$ is uniquely specified in the above two models.

For a stationary LDOUP $\bfX$, let $\wt{\bfvtheta}$ be the parameter vector of the BDLP $\bfZ$ if $\bfX$ is specified in terms of the BDLP as in model \ref{model1} or, of the stationary distribution $\bfY$ as in model \ref{model2}. Thus, the law of $\bfX$ is determined by the parameter vector $\bfvtheta:=(\lambda,\wt{\bfvtheta})$. Let $t_0=0,t_1=\Delta , \dots ,t_m= m\Delta$, be equally spaced observation times with sampling interval $\Delta>0$.

For $\Delta >0 $, define the random vector
\begin{align}\label{zstar}
    \bfZ^*(\Delta):=  \int_0^{\lambda\Delta}e^{ s}\,\rmd\bfZ( s).
\end{align}
This innovation term plays a critical role throughout our analysis.

\section{LDOUP Using Weak Subordination}\label{sec4}

In this section, we consider a LDOUP where the stationary distribution or BDLP is a weakly subordinated L\'evy process, specifically a WVAG process with drift. This produces the WVAG-OU and OU-WVAG processes, respectively.

When the stationary distribution is a WVAG process, we show that the innovation term $\bfZ^*(\Delta)$ is a discrete and continuous mixture, which is not absolutely continuous. This representation allows for the exact simulation of the corresponding LDOUP. We give the likelihood function in Corollary \ref{likwvagou} below. 

When the BLDP is a WVAG process, we show that $\bfZ^*(\Delta)$ is absolutely continuous, and consequently, we give the likelihood function in Proposition \ref{likprop} in the next section, but there is no known closed-form solution for even the characteristic exponent of $\bfZ^*(\Delta)$.

Throughout this section, let $\bfW\sim WVAG^n(a,\bfalpha,\bfmu,\Sigma,\bfeta)$ for $n\geq 2$.


\subsection{WVAG-OU Process}\label{wvagousec}

Recall that $\bfY\eqd \bfW(1)$ is a stationary distribution of a LDOUP if and only if $\bfW\sim SD^n$. By \cite[Corollary 4.4]{BLM17b}, a sufficient condition for this is $\bfmu=\bfnull$, and this is also necessary provided $\Sigma$ is invertible. This leads to the following definition.

\begin{definition}\label{defn1}
    A process $\bfX\sim WVAG^n\text{-}OU(\lambda,a,\bfalpha,\Sigma,\bfeta)$ is a \emph{weak variance alpha-gamma Ornstein-Uhlenbeck (WVAG-OU) process} if it is a stationary LDOUP with autocorrelation parameter $\lambda>0$ and stationary distribution $\bfY\eqd\bfW(1)$, where $\bfW\sim WVAG^n(a,\bfalpha,\bfnull,\Sigma,\bfeta)$.
\end{definition}

Next, we give an explicit representation of the corresponding BDLP $\bfZ$ and innovation term $\bfZ^*(\Delta)$.

\begin{theorem}\label{wvagouchar}
    Suppose that $\bfX\sim WVAG^n\text{-}OU(\lambda,a,\bfalpha,\Sigma,\bfeta)$. Then the BDLP $\bfZ$ can be characterised in the following equivalent ways:
    \begin{enumerate}[(i)]
        \item $\bfZ$ has characteristic exponent
        \begin{align}\label{wvagouz}
            \Psi_{\bfZ}(\bftheta) = \rmi\skal{\bfeta}{\bftheta}-\frac{a \|\bftheta\|^2_{\bfalpha\tr\Sigma}}{1+\frac 12 \|\bftheta\|^2_{\bfalpha\tr\Sigma}}-\sum_{k=1}^n \frac{\beta_k \alpha_k\Sigma_{kk}\theta_k^2}{1+ \frac 12\alpha_k\Sigma_{kk}\theta_k^2};
        \end{align}
        
        \item $\bfZ\sim CP^n(b,\PPP,\bfeta)$ is a compound Poisson process with drift, where
        \begin{align*}
            b :={} & 2\left(  a + \sum_{k=1}^n \beta_k\right),\\
            \PPP := {}&  \frac{a}{a + \sum_{k=1}^n \beta_k} \PPP_{0} +\sum_{k=1}^n \frac{\beta_k}{a + \sum_{k=1}^n \beta_k} (\bfdelta_0^{\otimes (k-1)}\otimes\PPP_{k}\otimes\bfdelta_0^{\otimes (n-k)}),
        \end{align*}
        and $\PPP_{0}$, $\PPP_k$, $k=1,\dots,n$, are the probability laws of $VG^n(1,\bfnull,\allowbreak\bfalpha\tr\Sigma)$, $VG^1(1,0,\allowbreak\alpha_k\Sigma_{kk})$, respectively.
        
    \end{enumerate}
    Furthermore, $\bfZ^*(\Delta)$, $\Delta> 0$, can be characterised in the following equivalent ways:
    \begin{enumerate}
        \item[(iii)] $\bfZ^*(\Delta)$ has characteristic exponent
        \begin{align}
            \begin{split}
                \Psi_{\bfZ^*(\Delta)}(\bftheta) = {}&\rmi\skal{\bfeta}{\bftheta}(e^{\lambda\Delta}-1)-a\log\left(\frac{1+\frac 12 \|\bftheta\|^2_{\bfalpha\tr\Sigma} e^{2\lambda\Delta}} {1+\frac 12 \|\bftheta\|^2_{\bfalpha\tr\Sigma}} \right) \\&-\sum_{k=1}^n \beta_k\log \left(\frac{1+ \frac 12\alpha_k\Sigma_{kk}\theta_k^2e^{2\lambda\Delta}}{1+ \frac 12\alpha_k\Sigma_{kk}\theta_k^2} \right) ;\label{wvagouzdelta}
            \end{split}
        \end{align}
        
        \item[(iv)] 
        \begin{align}\label{zstarcprep}
            \bfZ^*(\Delta) \eqd \bfeta(e^{\lambda\Delta}-1) + \sum_{l=1}^{N_0(\lambda\Delta)}e^{T_{0l}} \bfV_{0l} +  \sum_{k=1}^n\sum_{l=1}^{N_k(\lambda\Delta)}e^{T_{kl}} V_{kl}\bfe_k,
        \end{align}
        where $N_0\sim P_S(2a), N_k\sim P_S(2\beta_k)$, $\bfV_{0l}\sim VG^n(1,\bfnull,\bfalpha\tr\Sigma)$, $V_{kl}\sim VG^1(1,0,\alpha_k\Sigma_{kk})$, $k=1,\dots,n$, $l\in\NN$, are independent, while $T_{kl}$ is the $l$th arrival time of $N_k$, $k=0,1\dots,n$, $l\in\NN$.
    \end{enumerate}
\end{theorem}

\begin{remark}
    The random vector associated with the probability law $\PPP$ is
    \begin{align*}
        \begin{cases}
            \bfV_0(1) & \text{with probability $\frac{a}{a + \sum_{k=1}^n \beta_k}$},\\
            V_1(1)\bfe_1& \text{with probability $\frac{\beta_1}{a + \sum_{k=1}^n \beta_k}$},\\
            \vdots\\
            V_n(1)\bfe_n & \text{with probability $\frac{\beta_n}{a + \sum_{k=1}^n \beta_k}$},
        \end{cases}
    \end{align*}
    where $\bfV_0\sim VG^n(1,\bfnull,\bfalpha\tr\Sigma)$, $V_k\sim VG^1(1,0,\alpha_k\Sigma_{kk})$, $k=1,\dots,n$. Also, since these are VG distributions with $b=1$, they are also Laplace distributions. \qed
\end{remark}

\begin{remark}\label{simrem}
    For a general $\bfX\sim \operatorname{\mathit{OU-\bfZ}}(\lambda)$, the observations can be simulated using
    \begin{align}\label{obsx}
        \bfX(t_k) = e^{-\lambda\Delta}\left( \bfX(t_{k-1}) +\bfZ^*(\Delta)^{(k)}\right), \quad k =1,\dots, m,
    \end{align}
    where $( \bfZ^*(\Delta)^{(k)})_{k=1,\dots,m}$ are iid copies of $\bfZ^*(\Delta)$ (see Remark \ref{indepremark}).
    
    Thus, exact simulations of $\bfX\sim WVAG^n\text{-}OU(\lambda,a,\bfalpha,\Sigma,\bfeta)$ can be obtained using \eqref{obsx} and $\bfZ^*(\Delta)$ simulated with the explicit representation \eqref{zstarcprep}. Note that $\bfX(0)\eqd \bfY$ is the stationary distribution in Definition \ref{defn1}, which is a WVAG distribution. This can be simulated as a sum of independent VG random variables by \cite[Remark 3.9]{BLM17a}, which says that $\bfW\sim WVAG^n(a,\bfalpha,\bfmu,\Sigma,\bfeta)$ has explicit representation
    \begin{align}\label{wvagpropb}
        \bfW\eqd \bfeta I+ \bfV_0+(V_1,\dots,V_n),
    \end{align}
    where $\bfV_0\sim VG^n( a,a\bfalpha\tr\bfmu,a\bfalpha\tr\Sigma)$, $V_k\sim VG^1(\beta_k,\alpha_k\beta_k\mu_k,\alpha_k\beta_k\Sigma_{kk})$, $k=1,\dots,n$, are independent. This simulation method is implemented in Section \ref{simresultsec} below. \qed
\end{remark}

From Theorem~\ref{wvagouchar} (iv), $\bfZ^*(\Delta)$ does not have a Lebesgue density since it has a strictly positive probability of taking the value $\bfeta(e^{\lambda\Delta}-1)$, and $\bfX$ is a finite activity process since $\bfZ$ is a compound Poisson process. In Corollaries \ref{wvagoudenlem} and \ref{likwvagou} below, we restrict to the case $n=2$. Define $\bfzeta=(\zeta_1,\zeta_2):= \bfeta(e^{\lambda\Delta}-1)$. Here, $\bfZ^*(\Delta)$ is a discrete and continuous mixture over four mutually singular measures. There is a strictly positive probability of each of the following cases: by time $\lambda\Delta$, none of $N_1,N_2,N_0$ jump, giving a degenerate random vector $\bfzeta$; only $N_1$ jumps, giving an absolutely continuous distribution in the first component and a degenerate random variable $\zeta_2$ in the second component; only $N_2$ jumps, giving a degenerate random variable $\zeta_1$ in the first component and an absolutely continuous distribution in the second component; or otherwise, giving an absolutely continuous distribution overall. The next lemma gives a Radon-Nikodym derivative of $\bfZ^*(\Delta)$.

\begin{corollary}\label{wvagoudenlem} Let $n=2$ and $\Sigma$ be invertible. Define
    \begin{align*}
        p :={}& e^{-2(a+\beta_1+\beta_2)\lambda\Delta},\\
        p_1 :={}&e^{-2a\lambda\Delta}(1-e^{-2\beta_1 \lambda\Delta})e^{-2\beta_2 \lambda\Delta},\\
        p_2 :={}&e^{-2a\lambda\Delta}e^{-2\beta_1 \lambda\Delta}(1-e^{-2\beta_2 \lambda\Delta}),\\
        p_0 :={}& 1-p-p_1-p_2.
    \end{align*}
    Define the subsets
    \begin{align*}
        S_1 :={}& \{(x_1,x_2)\in\RR^2:\text{$x_1 \neq \zeta_1$ and $x_2=\zeta_2$}  \} ,\\
        S_2 :={}& \{(x_1,x_2)\in\RR^2:\text{$x_1 = \zeta_1$ and $x_2 \neq \zeta_2$}  \},\\
        S_0 :={}&\RR^2 \setminus  \left(\{\bfzeta\}\cup S_1 \cup S_2 \right).
    \end{align*}
    There exist Lebesgue densities $f_1,f_2$ and $f_0$  corresponding respectively to the characteristic functions
    \begin{align}
        \Phi_k(\theta_k) ={}& \frac{  \displaystyle{\Bigg(\frac{1+ \frac 12\alpha_k\Sigma_{kk}\theta_k^2e^{2\lambda\Delta}}{1+ \frac 12\alpha_k\Sigma_{kk}\theta_k^2} \Bigg)^{-\beta_k}}-e^{-2\beta_k\lambda\Delta}}{1-e^{-2\beta_k\lambda\Delta}}, \quad k= 1,2,\label{psi12}\\
        \Phi_0(\bftheta) ={}& \frac{e^{-\rmi\skal{\bftheta}{\bfzeta}}\Phi_{\bfZ^*(\Delta)}(\bftheta)-p- p_1 \Phi_1(\theta_1)-p_2 \Phi_2(\theta_2)}{p_0},\quad \bftheta=(\theta_1,\theta_2)\in\RR^2.\label{psi0}
    \end{align}
    Furthermore, the density of $\bfZ^*(\Delta)$ with respect to the measure
    \begin{align}
        \nu:= \bfdelta_{\bfzeta}  + \LLL\otimes \bfdelta_{\zeta_2} + \bfdelta_{\zeta_1}\otimes \LLL + \LLL^2 \label{dommeas}
    \end{align}
    is
    \begin{align}
        f_{\bfZ^*(\Delta)}(\bfz):={}&\frac{\rmd \PPP_{\bfZ^*(\Delta)}}{\rmd \nu}(\bfz)\nonumber \\
        \begin{split}
            ={}& p \eins_{\{\bfzeta\}}(\bfz) + \sum_{k=1}^2 p_k f_{k}(z_k-\zeta_k) \eins_{S_k}(\bfz)\\
            &+ p_0 f_{0}(\bfz-\bfzeta) \eins_{S_0}(\bfz),\quad \bfz=(z_1,z_2)\in\RR^2.\label{zstarden}
        \end{split}
    \end{align}
\end{corollary}

In practice, the Lebesgue densities $f_1,f_2$ and $f_0$ in Corollary \ref{wvagoudenlem} are obtained by applying Fourier inversion to the respective characteristic functions.

\begin{remark}\label{exactrecover}
    In the representation given in Theorem~\ref{wvagouchar} (iv), if the Poisson processes do not jump by time $\lambda\Delta$, then the sample path of $\bfX$ is deterministic for some time, possibly allowing for the exact recovery of the parameters  $\lambda$ and $\bfeta$. 
    
    Let $(x_{1,0},\dots,x_{1,m})$ be the vector of observations of $(X_1(0),\dots,X_1(t_m))$. Consider a triplet of consecutive observations $x_{1,k},x_{1,k+1},x_{1,k+2}$. By \eqref{obsx}, if $Z^*_1(\Delta)=\eta_1(e^{\lambda\Delta}-1)$ in the time interval $[t_k,t_{k+2}]$ (there is no jump), then
    \begin{align*}
        x_{1,k+1} = e^{-\lambda\Delta}x_{1,k} + (1-e^{-\lambda\Delta})\eta_1,\quad
        x_{1,k+2} = e^{-\lambda\Delta}x_{1,k+1} + (1-e^{-\lambda\Delta})\eta_1.
    \end{align*}
    
    Therefore, to recover $\lambda$ and $\eta_1$, for every triplet of consecutive observations $x_{1,k},x_{1,k+1},\allowbreak x_{1,k+2}$, determine the slope $c$ and intercept $d$ of the line $y=cx+d$ passing the points $(x,y)=(x_{1,k},x_{1,k+1}), (x_{1,k+1},x_{1,k+2})$. Doing this over all $k=0,\dots,m-2$, whenever the same value of $c$ and $d$ occurs more than once, we must have $\lambda = -\log(c)/\Delta$ and $\eta_1 = d/(1-c)$. Note that if $Z_1^*(\Delta)\neq \eta_1(e^{\lambda\Delta}-1)$ in the time interval $[t_k,t_{k+2}]$ (there is at least one jump), then $Z_1^*(\Delta)$ is absolutely continuous so all other values of $c$ and $d$ are almost surely different.
    
    A similar method on other components $X_2,\dots,X_n$ can be used to recover $\eta_2,\dots,\eta_n$, respectively. Thus, given a sufficiently large $m$ or small $\Delta$, it is always possible to exactly recover $\lambda$ and $\bfeta$. This method is analogous to the exact recovery of $\lambda$ in the Gamma-OU process discussed in \cite[Section 4.1]{vstt}. \qed
\end{remark}

Recall that, in general, the likelihood function of $(\bfX(0),\dots, \bfX(t_m))$ where the law of $\bfX$ depends on a parameter vector $\bfvtheta$ is $\bfvtheta\mapsto {\rmd\PPP_{\bfvtheta}}/{\rmd \nu^*}$, the Radon-Nikodym derivative of the probability law $\PPP_{\bfvtheta}$ of $(\bfX(0),\dots, \bfX(t_m))$ with respect to a dominating measure $\nu^*$, where $\PPP_{\bfvtheta}$ is absolutely continuous with respect to $\nu^*$ for all possible values of $\bfvtheta$ (see \cite[Section 1.3.1]{sche95}). Ordinarily, the dominating measure is Lebesgue measure $\nu^*=\LLL^n$, but since $\bfZ^*(\Delta)$ is a discrete and continuous mixture, here it is more complicated.

Interestingly, to determine the likelihood function for the WVAG-OU process, it is crucial to assume $\lambda$ and $\bfeta$ are known because the measure in \eqref{dommeas} depends on $\lambda$ and $\bfeta$, so it is not clear how to find a dominating measure otherwise. We justify this assumption based on the ability to exactly recover $\lambda$ and $\bfeta$ as discussed above. We derive a likelihood function using the AR(1) structure of the observations.


\begin{corollary}\label{likwvagou}
    Let $n=2$ and $\Sigma$ be invertible. Suppose that $\bfX\sim WVAG^2\text{-}OU\allowbreak(\lambda,a,\bfalpha,\Sigma,\bfeta)$. Let $\bfx=(\bfx_0,\dots,\bfx_m)$ be the vector of observations of $(\bfX(0),\dots,\allowbreak\bfX(t_m))$. Assume that $\lambda$ and $\bfeta$ are known, so that the parameter vector is $\bfvtheta:=(a,\bfalpha,\Sigma)$. Then the likelihood function is $\bfvtheta\mapsto L(\bfvtheta,\bfx)$, where
    \begin{align}\label{wvagoulik}
        L(\bfvtheta,\bfx) = f_{\bfX(0)}(\bfx_0)\prod_{k=1}^m f_{\bfZ^*(\Delta)}(e^{\lambda\Delta}\bfx_k-\bfx_{k-1}),
    \end{align}
    $ f_{\bfX(0)}$ is the Lebesgue density of $\bfX(0)$ and $f_{\bfZ^*(\Delta)}$ is given in \eqref{zstarden}.
\end{corollary}


\begin{remark}
    Corollaries \ref{wvagoudenlem} and \ref{likwvagou} can be extended to the cases $n \geq 3$ using similar arguments. For instance, if $n=3$, by Theorem~\ref{wvagouchar} (iv), we would need to consider many additional cases such as $N_0(\lambda\Delta)=0, N_1(\lambda\Delta)>0, N_2(\lambda\Delta)>0, N_3(\lambda\Delta)=0$; and $N_0(\lambda\Delta)=0, N_1(\lambda\Delta)>0, N_2(\lambda\Delta)=0, N_3(\lambda\Delta)>0$, which would lead to additional mutually singular measures $\LLL^2\otimes \bfdelta_{\zeta_{3}}$ and $\LLL\otimes \bfdelta_{\zeta_{2}}\otimes\LLL$, and so on. \qed
\end{remark}

\begin{remark}\label{marginallem}
    Obviously, if $\bfX=(X_1,\dots,  X_n)\sim \operatorname{\mathit{OU-\bfZ}}(\lambda)$, where $\bfZ=(Z_1,\dots,Z_n)\sim L^n$, then $X_k\sim \operatorname{\mathit{OU-Z_k}}(\lambda)$, $k = 1,\dots, n$. This follows by examining the probability transition distribution of $\bfX$ (see  \cite[Lemma 17.1]{sato99}). \qed
\end{remark}

\begin{remark} \label{marglik}
    By Remark \ref{marginallem} and the fact that the marginal components of a WVAG process are VG processes (see \cite[Remark 3.1]{BLM17a}), it follows that $X_k\sim Y_k\text{-}OU(\lambda)$ is a VG-OU process, where $Y_k\sim VG^1(1/\alpha_k,\mu_k,\allowbreak\Sigma_{kk},\eta_k)$ and $\mu_k=0$. But as a univariate process, $X_k$ is a stationary LDOUP for all $\mu_k\in\RR$, and the restriction $\bfmu=\bfnull$ is only needed for the multivariate process $\bfX$ to be stationary. However, we do not need to deal with the $\mu_k\neq 0$ case here. 
    
    In Section \ref{estresultsec} below, we also need the likelihood function for the marginal components, where $(x_{k,0},\dots,x_{k,m}) $ is the the vector of observations of $(X_k(0),\dots,\allowbreak X_k(t_m))$, $k=1,2$. From Theorem \ref{wvagouchar} (ii), the BDLP of $X_k$ is $Z_k\sim CP^1(2/\alpha_k,\PPP_k,\eta_k)$, and using similar arguments as in the proof of Corollaries \ref{wvagoudenlem} and \ref{likwvagou}, the likelihood function here is  $\bfvtheta:=(\alpha_k,\Sigma_{kk})\mapsto L(\bfvtheta,\bfx) =f_{X_k(0)}(x_{k,0})\prod_{l=1}^m f_{Z^*_k(\Delta)}(e^{\lambda\Delta}x_{k,l}-x_{k-1,l})$, where
    \begin{align*}
        f_{Z^*_k(\Delta)}(z) = p\eins_{\{\zeta_k\}} + (1-p)\eins_{\{\zeta_k\}^C}(z)f(z-\zeta_k), \quad z\in\RR,
    \end{align*}
    $p:= e^{-(2/\alpha_k)\lambda\Delta}$ and $f$ is the Lebesgue density corresponding to the characteristic function
    \begin{align*}
        \Phi(\theta) ={}& \frac{  \displaystyle{\Bigg(\frac{1+ \frac 12\alpha_k\Sigma_{kk}\theta_k^2e^{2\lambda\Delta}}{1+ \frac 12\alpha_k\Sigma_{kk}\theta_k^2} \Bigg)^{-1/\alpha_k}}-p}{1-p},\quad \theta\in\RR.
    \end{align*}\qed
\end{remark}

\subsection{OU-WVAG Process}\label{ouwvagsec}

Next, consider the LDOUP where the WVAG process is used as the BDLP.

\begin{definition}
    A process $\bfX\sim OU\text{-}WVAG^n(\lambda,a,\bfalpha,\bfmu,\Sigma,\bfeta)$ is an \emph{Orn\-stein-Uhlenbeck weak variance alpha-gamma process (OU-WVAG)} if it is a stationary LDOUP with autocorrelation parameter $\lambda>0$ and BDLP $\bfZ\sim WVAG^n(a,\bfalpha,\bfmu,\Sigma,\bfeta)$.
\end{definition}

\begin{proposition}\label{ouwvagprop}
    Suppose $\bfZ\sim WVAG^n(a,\bfalpha,\bfmu,\Sigma,\bfeta)$. Then:
    \begin{enumerate}[(i)]
        \item  $\bfZ$ satisfies the log moment condition \eqref{logcond};
        \item $\bfZ^*(\Delta)$ is absolutely continuous.
    \end{enumerate}
\end{proposition}

Proposition~\ref{ouwvagprop} (i) implies the OU-WVAG process does in fact always have a stationary distribution. Proposition~\ref{ouwvagprop} (ii) implies that the likelihood function for the OU-WVAG process is the more typical case where the dominating measure is the Lebesgue measure. Also, the OU-WVAG process has infinite activity since its BDLP does.

\begin{remark}\label{cantextendrem}
    For the univariate OU-VG process, a formula for $\Psi_{Z^*(\Delta)}$ in terms of the dilogarithm function was found in \cite[Equation (18)]{Sab20} using the fact that the univariate VG process can be decomposed as a difference of two independent gamma processes (see Remark~\ref{vgdecomp}). But for the multivariate OU-VG process, and hence the OU-WVAG process, this method cannot be generalised as $VG^n(b,\bfmu,\Sigma)$, $n\geq2$, is not in the class of generalised gamma convolutions on $\RR^n$ when $\Sigma$ is invertible (see \cite[page 2220]{BKMS16}). This explains why we cannot find a result similar to Theorem \ref{wvagouchar} for the OU-WVAG process. \qed
\end{remark}

The likelihood function and exact simulation method in this section for the WVAG-OU process cannot be adapted to the OU-WVAG process, so the next section gives estimation and simulation methods that can be used for the OU-WVAG process, though they are also applicable much more generally.

\section{Estimation and Simulation for LDOUP}\label{sec3}

For a general $\bfX\sim \operatorname{\mathit{OU-\bfZ}}(\lambda)$, in this section, we give a likelihood function for the maximum likelihood estimation of the parameter vector $\bfvtheta$ assuming $\bfZ^*(\Delta)$ is absolutely continuous, and consider the simulation of the observations $\bfX(0),\bfX(t_1),\dots,\allowbreak\bfX(t_m)$ using an Euler scheme approximation of $\bfZ^*(\Delta)$. At the end we give moment formulas for $\bfZ^*(\Delta)$.

Both the estimation and simulation parts here follow the univariate results of \cite{vstt} closely but generalised to higher dimensions, and there is no surprise the results extend very straightforwardly.

\subsection{Likelihood Function and Estimation}\label{liksec}

We have the following likelihood function when $\bfZ^*(\Delta)$ is absolutely continuous.

\begin{proposition}\label{likprop}
    Let $\bfX\sim \operatorname{\mathit{OU-\bfZ}}(\lambda)$ be a stationary LDOUP. Let $\bfx=(\bfx_0,\dots,\bfx_m)$ be the vector of observations of $(\bfX(0),\dots,\bfX(t_m))$. Assume that $\bfX(0)$ is nondegenerate and $\bfZ^*(\Delta)$ is absolutely continuous. Then the likelihood function is $\bfvtheta\mapsto L(\bfvtheta,\bfx)$, where
    \begin{align}\label{likfn}
        L(\bfvtheta,\bfx) = e^{mn\lambda\Delta}f_{\bfX(0)}(\bfx_0)\prod_{k=1}^m f_{\bfZ^*(\Delta)}(e^{\lambda\Delta}\bfx_k-\bfx_{k-1}),
    \end{align}
    and $ f_{\bfX(0)}$ and $ f_{\bfZ^*(\Delta)}$ are the Lebesgue densities of $\bfX(0)$ and $\bfZ^*(\Delta)$, respectively.
\end{proposition}

\begin{remark}
    If we fix $\bfX(0) = \bfx_0$ instead of assuming it has a nondegenerate stationary distribution $\bfY$, so that the LDOUP $\bfX$ is, in general, no longer stationary, then using similar arguments as in the proof,  $L(\bfvtheta,\bfx) = e^{mn\lambda\Delta}\prod_{k=1}^m f_{\bfZ^*(\Delta)}(e^{\lambda\Delta}\bfx_k\allowbreak-\bfx_{k-1})$. Note that if $\bfX(0) = \bfx_0$, then $\bfX$ remains stationary only in the degenerate case where $\bfY =\bfx_0$, $\bfX(t)=\bfx_0$, $\bfZ(t)= \bfx_0 t$, $t\geq 0$, almost surely.\qed
\end{remark}

Under the absolute continuity assumption, the above result gives the following general procedure for estimating the parameter vector $\bfvtheta$ of the LDOUP $\bfX$ using ML.
\begin{enumerate}[1.]
    \item Compute the likelihood function  $L(\bfvtheta,\bfx)$, where the Lebesgue densities $f_{\bfX(0)}$ and $f_{\bfZ^*(\Delta)}$ correspond respectively to the characteristic exponents
    \begin{align}
        \Psi_{\bfY}(\bftheta) ={}& \int_{0}^\infty \Psi_\bfZ (e^{-t}\bftheta)\,\rmd t, \label{cey} \\
        \Psi_{\bfZ^*(\Delta)}(\bftheta) ={}& \int_{0}^{\lambda\Delta} \Psi_\bfZ (e^{ t}\bftheta)\,\rmd t,\quad \bftheta\in\RR^n.  \label{cez}
    \end{align}
    \item Find the parameter vector $\bfvtheta$ which maximises $L(\bfvtheta,\bfx)$.
\end{enumerate}

The characteristic exponents \eqref{cey} and \eqref{cez} follow from Lemmas \ref{conversionlem} and \ref{cor1}, respectively, and in practice, the Lebesgue densities are obtained by applying Fourier inversion to the characteristic functions. Applying this procedure without first checking that the absolute continuity assumption holds may lead to biased estimates. We have shown this condition does not hold for the WVAG-OU process but does for the OU-WVAG process. Conditions for consistency and asymptotic normality of ML in the dependent observations case can be found in, for example, \cite{HM1,HM2}.

\begin{remark}\label{compburd}
    For the OU-WVAG process, when numerically computing $f_{\bfZ^*(\Delta)}$ in the likelihood function, taking the Fourier inversion of $\Phi_{\bfZ^*(\Delta)}$ requires evaluating $\Psi_{\bfZ^*(\Delta)}$ in \eqref{cez} using numerical integration. Thus, it would significantly reduce the computational burden if there is a closed-form formula for $\Psi_{\bfZ^*(\Delta)}$ like in Theorem \ref{wvagouchar} (iii), but we cannot go beyond \eqref{cez} as explained in Remark \ref{cantextendrem}.\qed
\end{remark}

\subsection{Simulation}

Let $\bfX\sim \operatorname{\mathit{OU-\bfZ}}(\lambda)$. We consider the simulation of the vector of observations
\begin{align}
    (\bfX(0),\bfX(t_1),\dots, \bfX(t_m)).\label{euler}
\end{align}
Set $\bfX(0) = \bfX_0 \eqd \bfY$. Suppose the stationary distribution $\bfY$ and the BDLP $\bfZ$ can be simulated.


Exact simulation of \eqref{euler} is possible using \eqref{obsx} provided that $\bfZ^*(\Delta)$ can be exactly simulated (for example, Remark \ref{simrem}). Otherwise, to simulate $\bfZ^*(\Delta)\eqd\int_0^{\lambda\Delta} e^s\,\rmd\bfZ(s)$, for a small step size $\wt\Delta>0$, we can use the stochastic integral approximation
\begin{align}\label{discretez}
    \bfZ^*_{\wt\Delta}(\Delta) :={}\sum_{l=1}^{\wt n} e^{l\wt\Delta} (\bfZ( (l+1)\wt\Delta )-\bfZ(l\wt\Delta)),
\end{align}
where $\wt n =   \lfloor  \lambda\Delta/\wt\Delta  \rfloor$. Let $(\bfZ(\wt\Delta)^{(k,l)})_{k=1,\dots,m,\,l=1,\dots,\wt n}$ be iid copies of $\bfZ(\wt\Delta)$ and $\bfZ^*_{\wt\Delta}(\Delta)^{(k)}:= \sum_{l=1}^{\wt n} e^{l\wt\Delta} \bfZ(\wt\Delta)^{(k,l)}$, so  that $(\bfZ^*_{\wt\Delta}(\Delta)^{(k)})_{k=1,\dots,m}$ are iid copies of $\bfZ^*_{\wt\Delta}(\Delta)$. Thus, \eqref{euler} can be simulated using the Euler scheme approximation $(\bfX_{\wt\Delta}(0), \bfX_{\wt\Delta}(\Delta),\allowbreak\dots,\bfX_{\wt\Delta}(m\Delta))$, where $\bfX_{\wt\Delta}(0):=\bfX_0$ and
\begin{align}\label{approxsim}
    \bfX_{\wt\Delta}(t_k) := {}e^{-\lambda\Delta} \left( \bfX_{\wt\Delta}(t_{k-1}) +\bfZ^*_{\wt\Delta}(\Delta)^{(k)}\right) ,\quad k =1,\dots, m.
\end{align}


The convergence of this approximate simulation scheme is essentially proven by \cite[Theorem 3.1]{JP98} (though that result assumes $\bfX_0$ is deterministic, whereas here it is random) with the rate of convergence in the univariate case given in \cite[Theorem 6.1]{JP98} and \cite{jac04}. We include a simple and direct proof of the following convergence result.

\begin{proposition}\label{zconvglem1}
    Let $\Delta>0$ and $\bfZ\sim L^n$.
    \begin{enumerate}
        \item[(i)] As $\wt\Delta\to 0$, $\bfZ^*_{\wt\Delta}(\Delta)\stp \bfZ^*(\Delta)$.
        \item[(ii)] Let $\bfX\sim \operatorname{\mathit{OU-\bfZ}}(\lambda)$ be a stationary LDOUP. As $\wt\Delta\to 0$, $(\bfX_{\wt\Delta}(0), \bfX_{\wt\Delta}(t_1),\allowbreak\dots,\bfX_{\wt\Delta}(t_m))\std(\bfX(0), \bfX(t_1),\dots, \bfX(t_m)) $.
    \end{enumerate}
\end{proposition}

\begin{remark}
    Throughout this section it is assumed that the observations are equally spaced with sampling interval $\Delta>0$. For possibly unequal sampling intervals $\Delta_{k} = t_{k}-t_{k-1}$, $k=1,\dots,m$, the above results can be extended in the obvious way, replacing $\Delta$ with $\Delta_{k}$ where appropriate. \qed
\end{remark}

\subsection{Moments}

Next, we give the moments of the innovation term $\bfZ^*(\Delta)$ in terms of the moments of $\bfZ$, and the autocorrelation of the stationary LDOUP $\bfX$. This will be useful in Section \ref{sec5}. For a random vector $\bfU=(U_1,\dots,U_n)$, let $m_1(\bfU):=\EE[\bfU]$, and the $k$th central moment of $\bfU$ be $m_k(\bfU):=(\EE[(U_1-\EE[U_1])^k],\dots,\EE[(U_n\allowbreak-\EE[U_n])^k])$, $k\geq2$, and also let $\gamma_k:=(e^{k\lambda\Delta}-1)/k$. For a stationary process $\bfX=(X_1,\dots,X_n)$, define its autocorrelation function at lag $t\geq0$ to be $\rho_\bfX(t):=(\myCorr (X_k(0),X_l(t))\in\RR^{n\times n}$, $t\ge0$.

\begin{lemma}\label{momprop}
    Let $\Delta>0$ and $\bfZ=(Z_1,\dots,Z_n)\sim L^n$. 
    
    \begin{enumerate}
        \item[(i)] The moments of $\bfZ^*(\Delta)=(Z^*_1(\Delta),\dots,Z^*_n(\Delta))$ include
        \begin{align}
            m_1(\bfZ^*(\Delta))&=\gamma_1m_1(\bfZ(1)),\label{mom1}\\
            m_2(\bfZ^*(\Delta))&=\gamma_2m_2(\bfZ(1)),\nonumber\\
            m_3(\bfZ^*(\Delta))&=\gamma_3m_3(\bfZ(1)),\nonumber\\
            m_4(\bfZ^*(\Delta))&=\gamma_4m_4(\bfZ(1))-3\gamma_2m_2(\bfZ(1))^2,\nonumber\\
            \myCov(Z^*_k(\Delta),Z^*_l(\Delta))&=\gamma_2\myCov(Z_k(1),Z_l(1)),\quad k\neq l,\label{mom2}
        \end{align}
        provided that the moments of $\bfZ$ on the RHS are finite.
        
        \item[(ii)] Let $\bfX\sim \operatorname{\mathit{OU-\bfZ}}(\lambda)$ be a stationary LDOUP and suppose its initial value $\bfX_0$ has autocorrelation matrix $P$. If $\bfX(t)$, $t\geq0$, has finite second moments, then the autocorrelation function of $\bfX$ is $\rho_\bfX(t)=P\exp(-\lambda t)$, $t\ge0$.
    \end{enumerate}
\end{lemma}

\begin{remark} \label{acfrem} In Lemma \ref{momprop} (ii), every marginal component $X_k$ has the same autocorrelation function $t\mapsto e^{-\lambda t}$. It is implicitly assumed that $\bfX_0$ is nondegenerate with finite second moments. If, further, $\bfZ$ has finite second moments, then so does $\bfZ^*(\Delta)$ and $\bfX(t)$ by Lemma \ref{momprop} (i) and \eqref{ousoln}, respectively.\qed
\end{remark}

We apply the moment formulas to the WVAG-OU and OU-WVAG processes.

\begin{corollary}\label{wvagoumom}
    If $\bfX\sim WVAG^n\text{-}OU(\lambda,a,\bfalpha,\Sigma,\bfeta)$, then the moments of $\bfZ^*(\Delta)=(Z^*_1(\Delta),\dots,\allowbreak Z^*_n(\Delta))$ include
    \begin{align*}
        m_1(Z^*_k(\Delta))&=\gamma_1\eta_k,\\
        m_2(Z^*_k(\Delta))&=2\gamma_2\Sigma_{kk},\\
        m_3(Z^*_k(\Delta))&=0,\\
        m_4(Z^*_k(\Delta))&=12\Sigma_{kk}^2(\gamma_4(\alpha_k+1)  -\gamma_2),\quad k=1,\dots,n,\\
        \myCov(Z^*_k(\Delta),Z^*_l(\Delta))&=2\gamma_2(\alpha_k\wedge\alpha_l )\Sigma_{kl},\quad k\neq l.
    \end{align*}
\end{corollary}

\begin{remark}\label{ouwvagmom}
    If $\bfX\sim OU\text{-}WVAG^n(\lambda,a,\bfalpha,\bfmu,\Sigma,\bfeta)$, the formulas for $\bfZ^*(\Delta)$ follow immediately from Lemma \ref{momprop} (i) and the moments of the WVAG process (see \cite[Remark 4 and Appdendix A.1]{MiSz17}). \qed
\end{remark}

\begin{remark}\label{covrem}
    For all $t\geq 0$ and $k\neq l$, if $\bfX\sim WVAG^n\text{-}OU(\lambda,a,\bfalpha,\Sigma,\bfeta)$, then $\myCov(X_k(t),\allowbreak X_l(t))=C(0,0)$, where  $C(\mu_k,\mu_l) := a((\alpha_k\wedge \alpha_l)\Sigma_{kl}+\alpha_k\alpha_l\mu_k\mu_l)$,  by \cite[Remark 4]{MiSz17}, and if $\bfX\sim OU\text{-}WVAG^n\allowbreak(\lambda, a,\bfalpha,\bfmu,\Sigma,\bfeta)$, then $\myCov(X_k(t),X_l(t))= C(\mu_k,\mu_l)/2$ by \eqref{obsx} and \eqref{mom2}. \qed
\end{remark}

\section{Numerical Results}\label{sec5}

In this section, we apply the previous results to simulate and estimate the parameters of a WVAG-OU process and an OU-WVAG process with $n=2$ using maximum likelihood.

Throughout, we set
\begin{gather*}
    \lambda = 0.5,\quad a =1, \quad \bfalpha =(\alpha_1,\alpha_2)= (0.9,0.5),\quad \bfmu =(\mu_1,\mu_2) =(0.15,-0.06),\\
    \quad \Sigma = \begin{pmatrix}\Sigma_{11}&\Sigma_{12}\\\Sigma_{12}&\Sigma_{22} \end{pmatrix}= \begin{pmatrix}0.18&0.09\\0.09&0.08 \end{pmatrix},\quad \bfeta =(\eta_1,\eta_2) =(-0.06,0).
\end{gather*}

With these as the true parameters, consider the process $\bfX\sim WVAG^2\text{-}\allowbreak OU(\lambda,a,\bfalpha,\Sigma,\bfeta)$ or $\bfX\sim OU\text{-}WVAG^2(\lambda,a,\bfalpha,\bfmu,\Sigma,\bfeta)$. Note that in the former case, there is no $\bfmu$ parameter since the stationary distribution is a WVAG distribution with $\bfmu=\bfnull$, not $\bfmu=(0.15,-0.06)$.

\begin{sloppypar}
    The  R code for this section can be found at \url{https://github.com/klu5893/LDOUP-Calibration}
\end{sloppypar}

\subsection{Simulations}\label{simresultsec}

\paragraph{WVAG-OU process} 
Let $\bfX\sim WVAG^2\text{-}OU(\lambda,a,\bfalpha,\Sigma,\bfeta)$. We make an exact simulation of $\bfX$ as explained in Remark \ref{simrem}.

Figure \ref{plot1} shows this simulation with $\Delta = 1/100$ for $t\in[0,1000]$, and the same sample path for $t\in[0,10]$. Zooming in, we see that if the Poisson processes in the representation of $\bfZ^*(\Delta)$ in \eqref{zstarcprep} do not jump by time $\lambda\Delta$, then the sample path is deterministic as explained in Remark \ref{exactrecover}.

\begin{figure}[!h]
    \begin{center}
        \includegraphics[scale=.75]{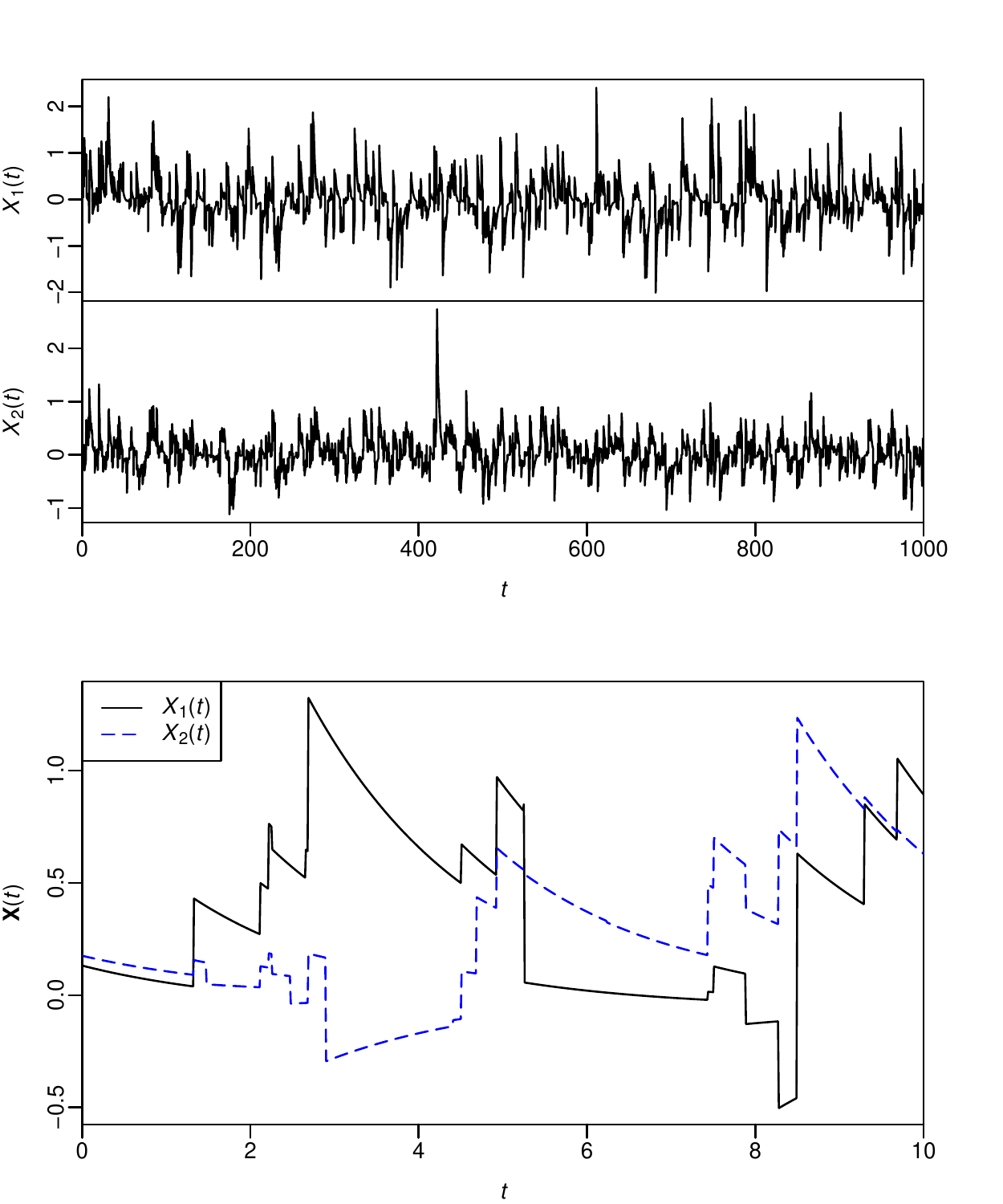}
        \caption{Simulated sample path of $\bfX\sim WVAG^2\text{-}OU(\lambda,a,\bfalpha,\Sigma,\bfeta)$ on $t\in[0,1000]$ (top) and the same sample path zoomed in to $t\in[0,10]$ (bottom).}
        \label{plot1}
    \end{center}
\end{figure}

\paragraph{OU-WVAG process}
Now let $\bfX\sim OU\text{-}WVAG^2(\lambda,a,\bfalpha,\bfmu,\Sigma,\bfeta)$. We simulate $\bfX$ using the approximation \eqref{approxsim} since we do not have an explicit representation for  $\bfZ^*(\Delta)$. Here, $\bfX(0)\eqd \bfY$ is the stationary distribution characterised by \eqref{cey}, which is simulated by numerically using Fourier inversion (see Section~\ref{estresultsec} below for details) to get the corresponding Lebesgue density and then sampling from that. Also, using \eqref{approxsim} requires simulating the WVAG distribution $\bfZ(\wt\Delta)$, which is done as explained in Remark \ref{simrem}.

Figure \ref{plot2} shows this simulation with $\Delta = 1/100$, $\wt\Delta = 1/10000$ for $t\in[0,1000]$, and the same sample path for $t\in[0,10]$. This is an infinite activity process, and while Figures \ref{plot1} and \ref{plot2} look somewhat similar, there are infinitely many extremely small jumps in the latter case.

\begin{figure}[!h]
    \begin{center}
        \includegraphics[scale=.75]{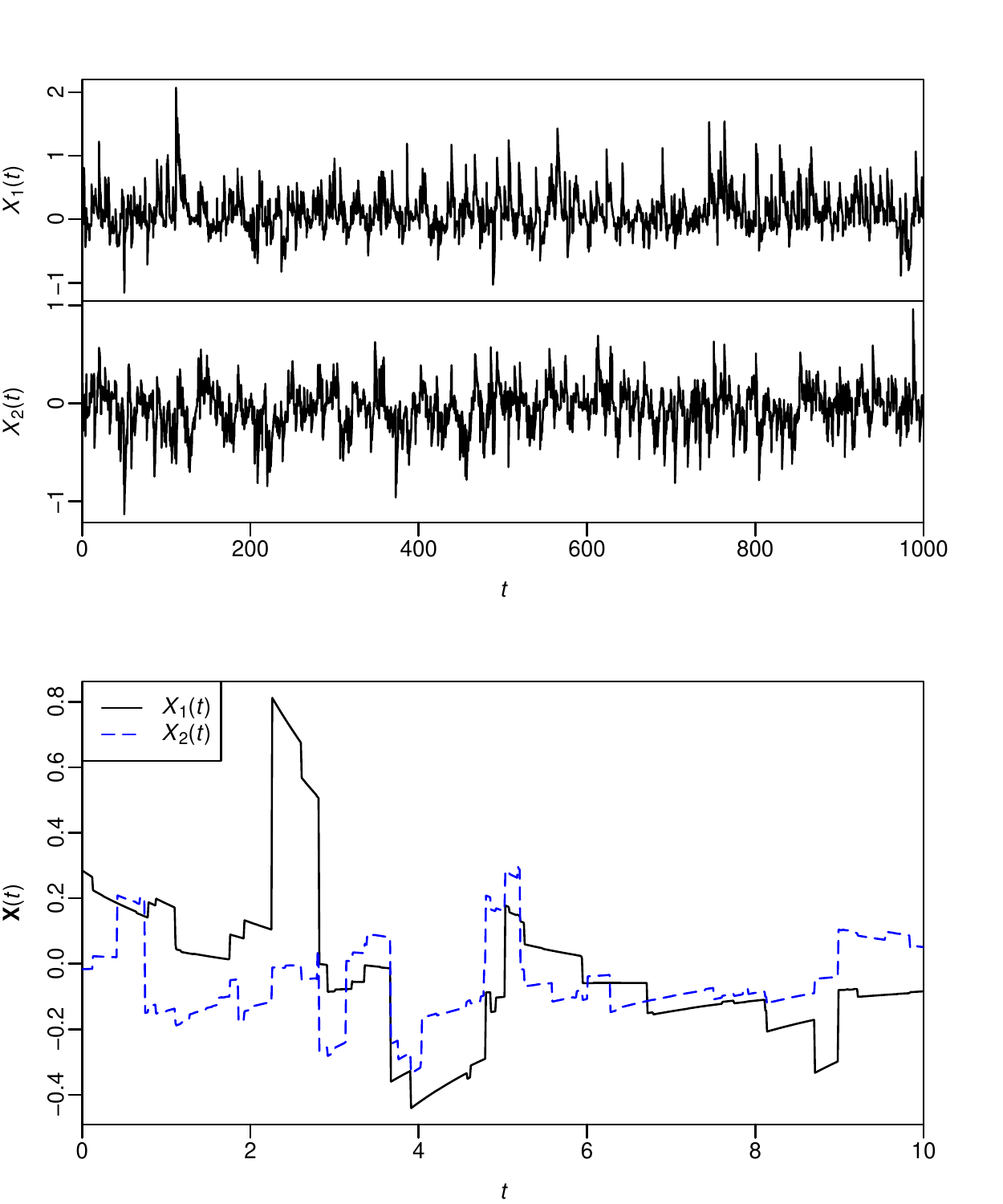}
        \caption{Simulated sample path of $\bfX\sim OU\text{-}WVAG^2(\lambda,a,\bfalpha,\bfmu,\Sigma,\bfeta)$ on $t\in[0,1000]$ (top) and the same sample path zoomed in to $t\in[0,10]$ (bottom).}
        \label{plot2}
    \end{center}
\end{figure}

These sample paths can be compared to those of Gamma-OU and OU-Gamma processes. For instance, see \cite[Figure 1]{QDZ19}. In fact, the marginal components $X_1,X_2$ can be obtained from the difference of independent Gamma-OU or OU-Gamma processes, though the joint process $\bfX$ cannot be.

\paragraph{Simulation checks} Now fix the sampling interval $\Delta=1$ and let $m=1000$, and we make 10000 Monte Carlo simulations of the observations $\bfX(0),\bfX(t_1),\dots,\bfX(t_m)$. Here, we consider three checks to show the simulation methods above for the WVAG-OU and OU-WVAG processes are correct. The first two relate to the distribution of $\bfX(t)$ for a fixed time $t$, which is equal to the stationary distribution, and the last relates to the autocorrelation of $\bfX$ for a fixed sample path.

Firstly, we check that the marginal distributions of $X_1(1000)$ and $X_2(1000)$ are correct. Figure \ref{histsd} show the histograms of $X_1(1000)$ and $X_2(1000)$ with a sample size of 10000 as well as the corresponding theoretical Lebesgue density. To obtain this density, note that for the WVAG-OU process, $X_k(1000)\sim VG^1(1/\alpha_k,0,\Sigma_{kk},\eta_k)$, $k=1,2$, by Remark \ref{marglik}, while for the OU-WVAG process, it is determined by \eqref{cey} and computed numerically using Fourier inversion (see Section~\ref{estresultsec} below for details). Visually, the histogram and theoretical densities closely match in all cases. We also perform one-sample Kolmogorov-Smirnov tests with sample sizes 100, 1000 and 10000. This tests whether the samples of $X_1(1000)$ and $X_2(1000)$ come from their respective theoretical probability distribution. For the WVAG-OU process, we would expect that the p-value would not be significant over all sample sizes because the simulation method is exact, whereas for the OU-WVAG process, the p-value would not be significant for reasonably large sample sizes but would become significant at some sufficiently large sample size because the simulation method is approximate and any small difference will eventually become detectable. This is consistent with the results given in Table \ref{kstest} where we see that all the p-values are not significant except for the OU-WVAG process when the sample size is 10000. In that case, the p-value is close to being significant at the 10\% level for $X_1(1000)$ and significant at the 5\% level for $X_2(1000)$. Therefore, we conclude that the marginal distributions are correct for the WVAG-OU process and highly accurate for the OU-WVAG process.

\begin{figure}[!h]
    \begin{center}
        \includegraphics[scale=.75]{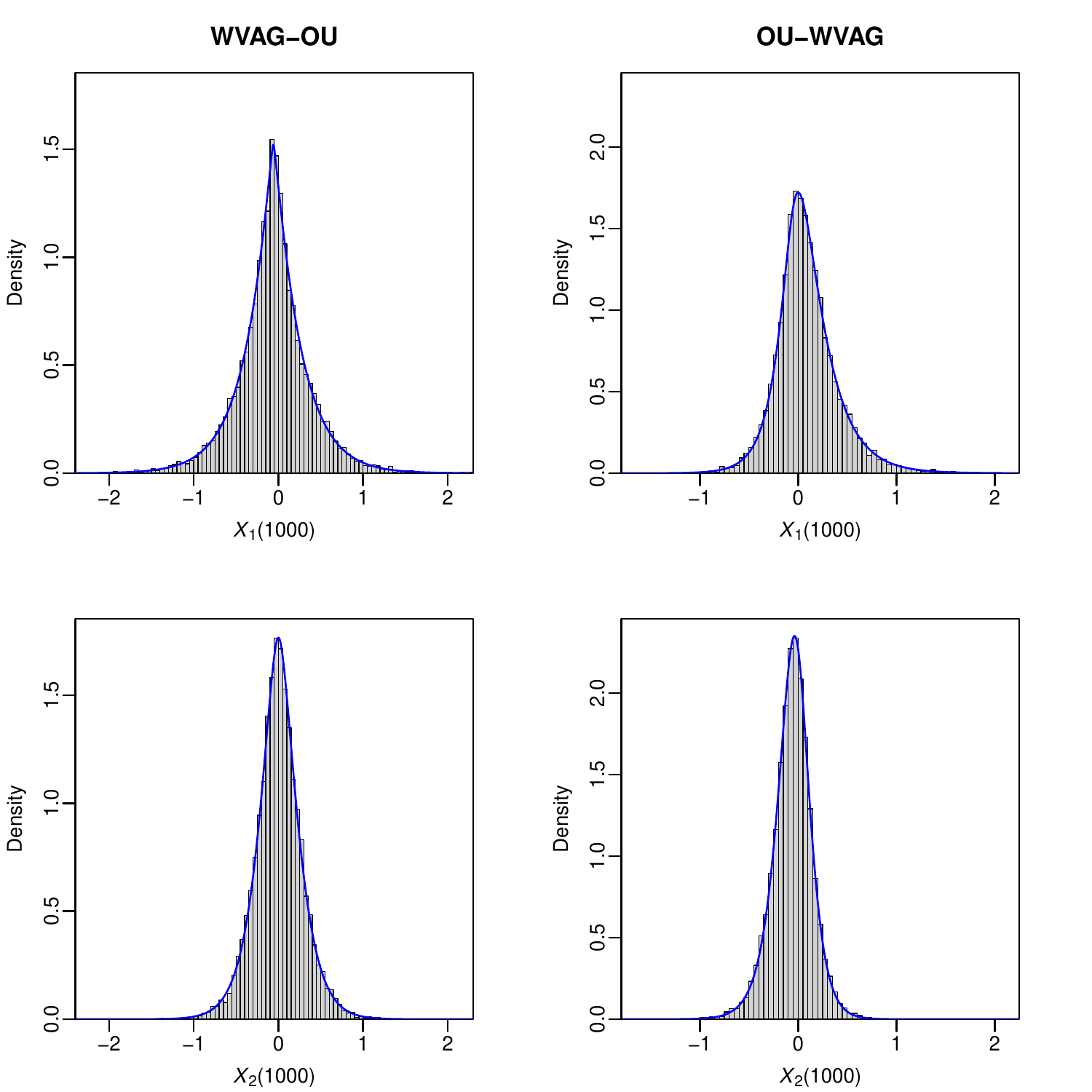}
        \caption{Histogram of $X_1(1000)$ (top) and $X_2(1000)$ (bottom) with a sample size of 10000, and corresponding theoretical Lebesgue density (blue line), where $\bfX\sim WVAG^2\text{-}OU(\lambda,a,\bfalpha,\Sigma,\bfeta)$ (left) or $\bfX\sim OU\text{-}WVAG^2(\lambda,a,\bfalpha,\bfmu,\Sigma,\bfeta)$ (right).}
        \label{histsd}
    \end{center}
\end{figure}

\begin{table}[h]
    \begin{center}
        \begin{tabular}{ccccc}
            \hline
            & \multicolumn{2}{c}{\bf WVAG-OU} & \multicolumn{2}{c}{\bf OU-WVAG}  \\
            {\bf Sample size} & $X_1(1000)$ & $X_2(1000)$ & $X_1(1000)$ &  $X_2(1000)$\\ \hline
            100          & 0.3869      & 0.9292 & 0.2637 & 0.9886\\
            1000         & 0.8254    & 0.4493 & 0.2483  &  0.1673\\
            10000       & 0.6581    & 0.5858 & 0.1149 &  0.0111\\ \hline
        \end{tabular}
    \end{center}
    \caption{P-values of one-sample Kolmogorov-Smirnov tests for $X_1(1000)$ and $X_2(1000)$, where $\bfX\sim WVAG^2\text{-}OU(\lambda,a,\bfalpha,\Sigma,\bfeta)$ or $\bfX\sim OU\text{-}WVAG^2(\lambda,a,\bfalpha,\bfmu,\Sigma,\bfeta)$, and the sample sizes are 100, 1000 and 10000.}
    \label{kstest}
\end{table}

Secondly, we check that the covariance $\myCov(X_1(1000),X_2(1000))$ is correct. Table \ref{covtable} gives the theoretical covariance (see Remark \ref{covrem}), the sample covariance with a sample size of 10000, and the 95\% confidence interval computed using bootstrap with 10000 replicates. For both the WVAG-OU and OU-WVAG processes, the confidence interval covers the theoretical covariance.

\begin{table}[h]
    \begin{center}
        \begin{tabular}{ccccc}
            \hline
            & {\bf WVAG-OU} &{\bf OU-WVAG}  \\ \hline
            Theoretical covariance          &0.0450              & 0.0205\\
            Sample covariance           & 0.0447              & 0.0212\\
            95\% confidence interval       & $(0.0416,0.0480)$   & $(0.0195,0.0229)$\\ \hline
        \end{tabular}
    \end{center}
    \caption{The theoretical covariance $\myCov(X_1(1000),X_2(1000))$, the sample covariance with a sample size of 10000, and the 95\% confidence interval computed using bootstrap with 10000 replicates,  where $\bfX\sim WVAG^2\text{-}OU(\lambda,a,\bfalpha,\Sigma,\bfeta)$ or $\bfX\sim OU\text{-}WVAG^2(\lambda,a,\bfalpha,\bfmu,\Sigma,\bfeta)$.}
    \label{covtable}
\end{table}

Thirdly, we check that the autocorrelation function of $X_1$ and $X_2$ are correct for one simulated sample path of $\bfX$. Figure \ref{acfsd} shows the theoretical autocorrelation function of $X_1$ and $X_2$, which is $t\mapsto e^{-\lambda t}$ by Remark \ref{acfrem}, and the sample autocorrelation function. In all cases, both the theoretical and sample autocorrelation functions closely match.

\begin{figure}[!h]
    \begin{center}
        \includegraphics[scale=.75]{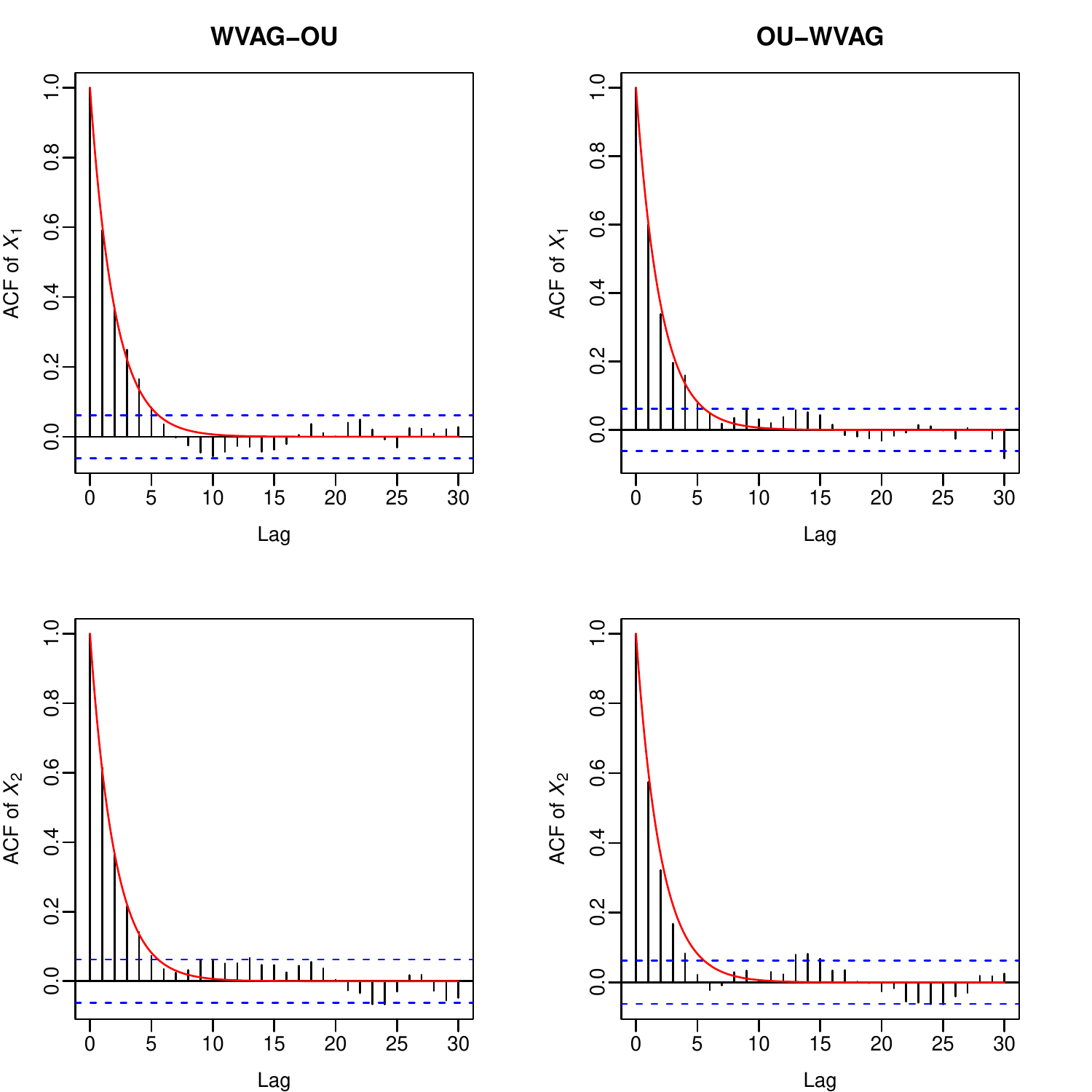}
        \caption{Theoretical (red line) and sample (black line) autocorrelation function for $X_1$ (top) and $X_2$ (bottom), where $\bfX\sim WVAG^2\text{-}OU(\lambda,a,\bfalpha,\Sigma,\bfeta)$ (left) or $\bfX\sim OU\text{-}WVAG^2(\lambda,a,\bfalpha,\bfmu,\Sigma,\bfeta)$ (right). The 95\% confidence interval around 0 is given by the dashed blue line.}
        \label{acfsd}
    \end{center}
\end{figure}

All three checks give the result we expect, which indicates that the simulation methods are correct.

\subsection{Estimation Method and Results}\label{estresultsec}

Recall that $\Delta=1$, $m=1000$, and we now make 500 Monte Carlo simulations of the observations $\bfX(0),\bfX(t_1),\dots,\bfX(t_m)$. In this section, we discuss the Monte Carlo results for simulating and estimating the parameters in the case where $\bfX\sim WVAG^2\text{-}OU(\lambda,a,\bfalpha,\Sigma,\bfeta)$ or $\bfX\sim OU\text{-}WVAG^2\allowbreak(\lambda,a,\bfalpha,\bfmu,\Sigma,\bfeta)$. The simulation is done as outlined in Section~\ref{simresultsec} with $\Delta=1$.

\paragraph{WVAG-OU process} 
Let $\bfX\sim WVAG^2\text{-}OU(\lambda,a,\bfalpha,\Sigma,\bfeta)$. 
We assume that $\lambda$ and $\bfeta$ are known, and are not estimated, as they can be exactly recovered by Remark~\ref{exactrecover}, similar to \cite[Table 1]{vstt}. Thus, there are 6 parameters $\bfvtheta = (a,\bfalpha,\Sigma)$ to be estimated using ML with the likelihood function in \eqref{wvagoulik}. Using the notation there, this requires computing the Lebesgue densities $f_{\bfX(0)}$, $f_1,f_2,f_0$ by taking the Fourier inversion of the characteristic functions corresponding to \eqref{v-alpha-g-cf} and \eqref{psi12}--\eqref{psi0}, respectively. Since the marginal components $X_1$, $X_2$ are also LDOUPs by Remark~\ref{marginallem}, we use the decoupled estimation method in \cite{BLM17c,MiSz17,LS10} as follows. The initial value of the parameters in the optimisations are set to the true value.

\begin{enumerate}
    \item The marginal parameters $(\alpha_k,\Sigma_{kk})$ are estimated using ML on the marginal observations $X_k(0),\dots,X_k(t_m)$, $k=1,2$.
    \item The full parameters $\bfvtheta$ are estimated using ML on the observations $\bfX(0),\dots,\allowbreak \bfX(t_m)$ with the initial values of $(\alpha_k,\Sigma_{kk})$, $k=1,2$, given by the estimates in Step 1.
\end{enumerate}

For Step 1, the likelihood function is given in Remark \ref{marglik}. The Fourier inversion method used is described in \cite[Section 4.1]{MiSz17}. Using the notation there, we set the number of grid points along each axis to be $N=2^{13}$ and the spacing to be $h_k=2^{-7}\sqrt{m_2(Z_k^*(\Delta))}$ for 1-dimensional Fourier inversion, and $N=2^{10}$ and $h_k=2^{-5}\sqrt{m_2(Z_k^*(\Delta))}$ for 2-dimensional Fourier inversion, for the components $k=1,2$.

\begin{table}[h]
    \begin{center}
        \begin{tabular}{ccccc}
            \hline
            {\bf Parameter/moment} & {\bf True value} & {\bf Mean estimate} & {\bf RMSE} & \\ \hline
            $a$                                    & 1      & 1.0014 & 0.0483\\
            $\alpha_1$                             & 0.9    & 0.9032 & 0.0382\\
            $\alpha_2$                             & 0.5    & 0.5008 & 0.0205\\
            $\Sigma_{11}$                          & 0.18   & 0.1810 & 0.0139\\
            $\Sigma_{22}$                          & 0.08   & 0.0799 & 0.0055\\
            $\Sigma_{12}$                          & 0.09   & 0.0903 & 0.0108\\\hline
            $m_2(Z_1^*(\Delta))$                   & 0.3093 & 0.3110 & 0.0238\\
            $m_4(Z_1^*(\Delta))$                   & 0.8459 & 0.8615 & 0.1300\\
            $m_2(Z_2^*(\Delta))$                   & 0.1375 & 0.1373 & 0.0094\\
            $m_4(Z_2^*(\Delta))$                   & 0.1180 & 0.1183 & 0.0163\\
            $\myCov(Z_1^*(\Delta),Z_2^*(\Delta))$  & 0.0773 & 0.0777 & 0.0092\\\hline
        \end{tabular}
    \end{center}
    \caption{Estimation results for the process $\bfX\sim WVAG^2\text{-}OU(\lambda,a,\bfalpha,\Sigma,\bfeta)$ based on 500 Monte Carlo simulations with $m=1000$.}
    \label{table1}
\end{table}

The estimation results are presented in Table \ref{table1}. For each parameter and the listed moments (see Corollary~\ref{wvagoumom}), we show the true value, mean estimate and RMSE over 500 Monte Carlo simulations. The estimated moments are calculated using the estimated parameters. We see that the estimated parameters are very accurate. Moreover, the fitted distribution is close to the true distribution in terms of the listed moments. The other moments, $m_1(\bfZ^*(\Delta))$ and $m_3(\bfZ^*(\Delta))$, are not listed since they do not depend on $\bfvtheta$ and so have no error.

\paragraph{OU-WVAG process} 
Now let $\bfX\sim OU\text{-}WVAG^2(\lambda,a,\bfalpha,\bfmu,\Sigma,\bfeta)$. There are 11 parameters $\bfvtheta = (\lambda,a,\bfalpha,\Sigma,\bfeta)$ to be estimated. In principle, the likelihood function \eqref{likfn} is applicable. However, there is an extremely heavy computational burden in maximising this as discussed in Remark~\ref{compburd}, so instead, we approximate the ML method with a stepwise procedure. Again, the initial value of the parameters in the optimisations are set to the true value.

\begin{enumerate}
    \item The parameter $\lambda$ is estimated by minimising the squared distance 
    \begin{align*}
        (e^{-\lambda\Delta}-\hat\rho_{X_1}(\Delta))^2 + (e^{-\lambda\Delta}-\hat\rho_{X_2}(\Delta))^2
    \end{align*}
    between the theoretical autocorrelation $e^{-\lambda\Delta}$ and sample autocorrelation $\hat\rho_{X_k}(\Delta)$ of lag $\Delta$ from the marginal observations $X_k(0),\dots,X_k(t_m)$, $k=1,2$.
    
    \item Given the estimate of $\lambda$, the marginal parameters $(\alpha_k,\mu_k,\Sigma_{kk},\eta_k)$ are estimated using ML on the marginal observations $X_k(0),\dots,X_k(t_m)$, $k=1,2$.
    \item Given the estimates of $(\lambda,\bfalpha,\bfmu,\Sigma_{11},\Sigma_{22},\bfeta)$, the joint parameters $(a,\Sigma_{12})$ are estimated using ML on the joint observations $\bfX(0),\dots, \bfX(t_m)$ subject to the constraint that the theoretical covariance $\myCov(Z^*_1(\Delta),Z^*_2(\Delta))$ in \eqref{mom2} and corresponding sample covariance matches.
\end{enumerate}

Step 1 is justified by Remark \ref{acfrem}, and similar to the method mentioned in \cite[page 9]{vstt}. Step 2 is justified by Remark~\ref{marginallem} as the LDOUP $X_k$ depends only on the marginal parameters $(\alpha_k,\mu_k,\Sigma_{kk},\eta_k)$ given $\lambda$. Steps 2 and 3 use the appropriate likelihood function in \eqref{likfn} with the ML method outlined in Section~\ref{liksec}. Note that the characteristic functions $\Phi_\bfY$ and $\Phi_{\bfZ^*(\Delta)}$ determined by \eqref{cey} and \eqref{cez} are evaluated using numerical integration. The details of the Fourier inversion method are outlined above. In Step 3, we are not using estimates of the previous steps as initial values, and combined with the constraint, this reduces the optimisation to a faster 1-dimensional problem, which is important given the high computational burden of evaluating the likelihood function. There are further optimisation constraints from the parameter restriction $|\Sigma_{12}/\sqrt{\Sigma_{11}\Sigma_{22}}|<1$ and $a\in(0,1/\alpha_1\wedge1/\alpha_2)$.

\begin{table}[h]
    \begin{center}
        \begin{tabular}{ccccc}
            \hline
            {\bf Parameter/moment} & {\bf True value} & {\bf Mean estimate} & {\bf RMSE} & \\ \hline
            $\lambda$                             & $\phantom{-}0.5$    &  $\phantom{-}0.5058$ & 0.0300\\
            $a$                                   & $\phantom{-}1$      &  $\phantom{-}0.9435$ & 0.1321\\
            $\alpha_1$                            & $\phantom{-}0.9$    &  $\phantom{-}0.9140$ & 0.0565\\
            $\alpha_2$                            & $\phantom{-}0.5$    &  $\phantom{-}0.5099$ & 0.0491\\
            $\mu_1$                               & $\phantom{-}0.15$   &  $\phantom{-}0.1548$ & 0.0230\\
            $\mu_2$                               & $-0.06$             &  $-0.0611$           & 0.0188\\
            $\Sigma_{11}$                         & $\phantom{-}0.18$   &  $\phantom{-}0.1840$ & 0.0175\\
            $\Sigma_{22}$                         & $\phantom{-}0.08$   &  $\phantom{-}0.0804$ & 0.0062\\
            $\Sigma_{12}$                         & $\phantom{-}0.09$   &  $\phantom{-}0.0946$ & 0.0179\\
            $\eta_1$                              & $-0.06$             & $-0.0621$            & 0.0092\\
            $\eta_1$                              & $\phantom{-}0 $     & $-0.0014$            & 0.0131\\\hline
            $m_1(Z_1^*(\Delta))$                  & $\phantom{-}0.0584$ &  $\phantom{-}0.0609$ & 0.0133\\
            $m_2(Z_1^*(\Delta))$                  & $\phantom{-}0.1720$ &  $\phantom{-}0.1802$ & 0.0185\\
            $m_3(Z_1^*(\Delta))$                  & $\phantom{-}0.0910$ &  $\phantom{-}0.1002$ & 0.0212\\
            $m_4(Z_1^*(\Delta))$                  & $\phantom{-}0.2949$ &  $\phantom{-}0.3299$ & 0.0717\\
            $m_1(Z_2^*(\Delta))$                  & $-0.0389$           & $-0.0413$            & 0.0099\\
            $m_2(Z_2^*(\Delta))$                  & $\phantom{-}0.0703$ &  $\phantom{-}0.0722$ & 0.0067\\
            $m_3(Z_2^*(\Delta))$                  & $-0.0085$           & $-0.0091$            & 0.0032\\
            $m_4(Z_2^*(\Delta))$                  & $\phantom{-}0.0315$ &  $\phantom{-}0.0339$ & 0.0069\\
            $\myCov(Z_1^*(\Delta),Z_2^*(\Delta))$ &$ \phantom{-}0.0352$ &  $\phantom{-}0.0355$ & 0.0062\\\hline
        \end{tabular}
    \end{center}
    \caption{Estimation results for the process $\bfX\sim OU\text{-}WVAG^2(\lambda,a,\bfalpha,\bfmu,\Sigma,\bfeta)$ based on 500 Monte Carlo simulations with $m=1000$.}
    \label{table2}
\end{table}

The estimation results are presented in Table \ref{table2}. The moments are calculated as described in Remark~\ref{ouwvagmom}. From the results, the estimated parameter and moments of the fitted distribution are reasonably accurate, though there are generally larger RMSEs and biases compared to the estimation of the WVAG-OU process, which is unsurprising. 

\subsection{Discussion}

This paper provides a general method for simulation and calibration using ML for multivariate LDOUPs, the latter assuming $\bfZ^*(\Delta)$ is absolutely continuous. For the WVAG-OU process, a likelihood function was derived taking into account that $\bfZ^*(\Delta)$ is a discrete and continuous mixture, and our Monte Carlo results show the estimation method is highly accurate.
In contrast, for the OU-WVAG process, the Monte Carlo results are less accurate given the approximate stepwise procedure used, which is needed due to the computational burden of not having a closed-form for $\Psi_{\bfZ^*(\Delta)}$, as explained in Remark \ref{compburd} and having to estimate 5 more parameters, yet estimation remains very slow. Using a desktop computer with a 3.0 GHz, 8-core processor and the \verb|doParallel| package in R, the total time to do the estimation over the 500 Monte Carlo simulations was 20,084 seconds for the WVAG-OU process and 351,515 seconds for the OU-WVAG process. Despite the difficulties, the OU-WVAG process may be a more realistic model than the WVAG-OU process for various stochastic phenomena as the latter has a deterministic sample path between jumps.

There are various research directions for LDOUPs beyond the scope of this paper. On the statistical side, ideas include formulating a LDOUP with infinite activity and a closed-form for $\Psi_{\bfZ^*(\Delta)}$
or proving asymptotic convergence results for the ML estimators studied here. In addition, one could consider more general models such as where the autocorrelation parameter $\lambda$ is replaced by a matrix as in \cite{Fas13} and sums of LDOUPs as in \cite{BNSh01}.


Some potential applications of our methods in mathematical finance include the multivariate modelling of short rates, volatilities, energy prices or price spreads using LDOUPs as noted in Section \ref{intro}. Moreover, our simulation method can be used for Monte Carlo option pricing when the risk-neutral dynamics of an energy price follow a known LDOUP. However, the calibration problem addressed here is easier than those that may occur in practice when the LDOUP is not directly observable. For instance, in the Barndorff-Nielsen and Shephard model where stochastic volatility is modelled by a LDOUP, the model is calibrated to stock prices or option prices. If the risk-neutral dynamics of an energy price follows a LDOUP with unknown parameters, then the model is calibrated to option prices. In both cases, parameters cannot be calibrated to the unobservable LDOUP. This problem has been studied in \cite{BNSh01,CKM18,GS06,PFH14}. Though  we make a contribution to calibration in the situation where the multivariate LDOUP is observed, extensions to deal with this problem are beyond our scope and an area open to future research.


\section{Proofs}\label{sec6}

\subsection{Useful Lemmas}

The following lemma states the form of the likelihood function that any Markov process whose value at the current observation time is some particular deterministic function of its value at the previous observation time and an independent innovation term should take. A straightforward application of this lemma implies Corollary \ref{likwvagou} and  Proposition \ref{likprop}.

\begin{lemma}\label{markovllf}
    
    For an $n$-dimensional Markov process $\bfX$, let $\bfx=(\bfx_0,\dots,\bfx_m)$ be the vector of observations of $(\bfX(0),\dots,\allowbreak\bfX(t_m))$ satisfying $\bfX(t_k) = M_{\bfX(t_{k-1})}(\bfZ^{(k)})$, $k=1,\dots,m$, where $\bfX(0),\bfZ^{(k)}$, $k=1,\dots,m$, are independent, each with law determined by the parameter vector $\bfvtheta$, $\bfX(0)$ has a Lebesgue density $f_{\bfX(0)}$ while each $\bfZ^{(k)}$ has density $f_k$ with respect to a measure $\nu_k$, and each $M_{\bfx_k} :\RR^n\to\RR^n$ is a bijective Borel measurable function defined by the mapping $\bfz\mapsto M(\bfx_k,\bfz,\bfvtheta)$ for some function $M$.
    If there exists a dominating measure $\nu^*$ such that $(\nu_m\circ M_{\bfx_{m-1}}^{-1})(\rmd \bfx_m)\dots (\nu_1  \circ M_{\bfx_{0}}^{-1})(\rmd \bfx_1)\rmd \bfx_0$ has a Radon-Nikodym derivative $g_{\bfvtheta}$ with respect to  $\nu^*$ for all possible values of $\bfvtheta$ (that is, $\nu^*$ cannot depend on $\bfvtheta$ but $g_{\bfvtheta}$ can), then the likelihood function is $\bfvtheta\mapsto L(\bfvtheta,\bfx)$, where
    \begin{align}\label{genlik}
        L(\bfvtheta,\bfx) =g_{\bfvtheta}(\bfx)f_{\bfX(0)}(\bfx_0)\prod_{k=1}^m (f_k\circ M_{\bfx_{k-1}}^{-1})(\bfx_{k}).
    \end{align}
\end{lemma}

\begin{proof}
    Let $\bfy_k=(y_{k1},\dots,y_{kn})\in\RR^n$ and $A_k := (-\infty,y_{k1}]\times \dots \times (-\infty,y_{kn}]$ for $k=0,1,\dots,m$. Interpreting the relation $\leq$ componentwise, consider the cumulative distribution function
    \begin{align}
        P:={}&\PP((\bfX(0),\dots,\bfX(t_m))\leq(\bfy_0,\dots,\bfy_m))\nonumber\\
        \begin{split}
            ={}&\int_{\RR^{n(m+1)}} \eins_{A_0}(\bfx_0)\dots\eins_{A_m}(\bfx_m)\,\PPP_{\bfX(t_m)\given\bfX(t_{m-1})=\bfx_{m-1}}(\rmd\bfx_m)\dots \\
            &\phantom{\int_{\RR^{n(m+1)}} \eins_{A_0}(\bfx_0)\dots\eins_{A_m}(\bfx_m)\,\,}  \PPP_{\bfX(t_1)\given\bfX(0)=\bfx_{0}}(\rmd\bfx_1)\PPP_{\bfX(0)}(\rmd \bfx_0) \label{itercdf}
        \end{split}
    \end{align}
    using the Markov property of $\bfX$. Note that $M_{\bfx_{k-1}}$ has inverse $M^{-1}_{\bfx_{k-1}}$, then we have
    \begin{align}
        \int_{A_k}\PPP_{\bfX(t_k)\given\bfX(t_{k-1})=\bfx_{k-1}}(\rmd\bfx_k) ={}& \PP(\bfX(t_k)\leq \bfy_k\given \bfX(t_{k-1})=\bfx_{k-1})\nonumber\\
        ={}& \PP(\bfZ^{(k)}\in  M^{-1}_{\bfx_{k-1}}(A_k))\nonumber\\
        ={}&\int_{\RR^n} \eins_{M^{-1}_{\bfx_{k-1}}(A_k)}(\bfz)f_k(\bfz)\,\nu_k(\rmd\bfz)\nonumber\\
        ={}&\int_{\RR^n} \eins_{A_k}(\bfx)(f_k\circ M^{-1}_{\bfx_{k-1}})(\bfx)\,(\nu_k\circ M^{-1}_{\bfx_{k-1}})(\rmd\bfx), \label{indivcdf}
    \end{align}
    where dropping the conditioning in the second line is justified as $\bfX(t_{k-1})$, being a function of $\bfX_0,\bfZ^{(1)},\dots, \bfZ^{(k-1)}$ only, is independent of $\bfZ^{(k)}$, and the last line follows from the transformation theorem \cite[Corollary 19.2]{Ba01}. Combining \eqref{itercdf} and \eqref{indivcdf} with Fubini's theorem gives
    \begin{align*}
        P={}&\int_{\RR^{n(m+1)}} \eins_{A_0}(\bfx_0)\dots\eins_{A_m}(\bfx_m)f_{\bfX(0)}(\bfx_0)\\
        &\phantom{\int_{\RR^{n(m+1)}} \,\,}\prod_{k=1}^m (f_k\circ M_{\bfx_{k-1}}^{-1})(\bfx_k)\,(\nu_m\circ M_{\bfx_{m-1}}^{-1})(\rmd \bfx_m)\dots(\nu_1  \circ M_{\bfx_{0}}^{-1})(\rmd \bfx_1)\rmd \bfx_0\\
        ={}&\int_{A_0\times\dots \times A_m}g_{\bfvtheta}(\bfx_0,\dots, \bfx_m)f_{\bfX(0)}(\bfx_0)\prod_{k=1}^m (f_k\circ M_{\bfx_{k-1}}^{-1})(\bfx_{k}) \,\nu^*(\rmd \bfx_0,\dots,\rmd \bfx_m),
    \end{align*}
    which implies the likelihood function is given by \eqref{genlik}. 
\end{proof}

The below lemma is based on \cite[pg 13]{vstt}.

\begin{lemma}\label{cppzlem}
    Let $\bfZ\sim CP^n(b,\PPP,\bfeta)$ be given by
    \begin{align}\label{cppdrift}
        \bfZ(t) = \bfeta t+ \sum_{k=1}^{N(t)} \bfJ_k,\quad t\geq0,
    \end{align}
    where $N\sim P_S(b)$ and $\bfJ_k$, $k\in\NN$, are iid with probability law $\PPP$. Then
    \begin{align*}
        \bfZ^*(\Delta) \eqd \bfeta(e^{\lambda \Delta}-1)+\sum_{k=1}^{N(\lambda \Delta)}e^{T_k}\bfJ_k,
    \end{align*}
    where $T_k$, $k\in\NN$, is the $k$th arrival time of $N$.
\end{lemma}

\begin{proof}
    For all $\wt\Delta>0$, introduce the simple function
    \begin{align}\label{simplefn}
        H_{\wt\Delta}(s) := \eins_{\{0\}}(s) +\sum_{k=1}^{\wt n+1} e^{(k-1)\wt\Delta}\eins_{((k-1)\wt\Delta, k \wt\Delta]}(s),
    \end{align}
    where $\wt n =   \lfloor  \lambda\Delta/\wt\Delta  \rfloor$.
    By the definition of the stochastic integral for simple functions (see \cite[Equation (2.4)]{saya}) and \eqref{cppdrift}, we have
    \begin{align*}
        \int_0^{\lambda\Delta} H_{\wt\Delta}(s)\,\rmd \bfZ(s) ={}& \sum_{k=0}^{\wt n}  e^{k\wt\Delta}(\bfZ((k+1)\wt\Delta) - \bfZ(k\wt\Delta))\\
        ={}& \sum_{k=0}^{\wt n} \bfeta e^{k\wt\Delta}\wt\Delta + \sum_{k=0}^{\wt n} e^{k\wt\Delta}\sum_{l \in S_k}\bfJ_l\\
        \stas {} & \bfeta\int_0^{\lambda\Delta}  e^{s} \,\rmd s + \sum_{l = 1}^{N(\lambda \Delta)}e^{T_l}\bfJ_{l}
    \end{align*}
    as $\wt\Delta \to 0$, where $S_k:= \{l=1,\dots,N(\lambda \Delta):T_l \in (k\wt\Delta,(k+1)\wt\Delta]\}$. Consequently, 
    \begin{align*}
        \int_0^{\lambda\Delta} H_{\wt\Delta}(s)\,\rmd \bfZ(s)\std \bfeta(e^{\lambda \Delta}-1)+\sum_{l=1}^{N(\lambda \Delta)}e^{T_l}\bfJ_l
    \end{align*}
    as $\wt\Delta \to 0$. However, \eqref{convgz} below gives  $\int_0^{\lambda\Delta} H_{\wt\Delta}(s)\,\rmd \bfZ(s)\std \bfZ^*(\Delta)$ as $\wt\Delta \to 0$. Since limits in distribution are unique, the result follows.  
\end{proof}  

\begin{lemma}\label{cpvglem}
    The compound Poisson process $\bfZ\sim CP^n(2a,\PPP)$, where $\PPP$ is the probability law of  $VG^n(1,\bfnull,\Sigma)$, has characteristic exponent
    \begin{align}
        \Psi_{\bfZ}(\bftheta) =-\frac{a \|\bftheta\|^2_{\Sigma}}{1+\frac 12 \|\bftheta\|^2_{\Sigma}},\quad \bftheta\in\RR^n.
    \end{align}
\end{lemma}

\begin{proof}
    The characteristic function of $\PPP$ is $\Phi_{\PPP}(\bftheta) = (1+\frac 12 \|\bftheta\|_\Sigma^2)^{-1}$ by \cite[Equation (2.9)]{BKMS16}, and the characteristic exponent of the compound Poisson process is $\Psi_{\bfZ}(\bftheta) = 2a(\Phi_{\PPP}(\bftheta)-1) $, from which the result follows.  
\end{proof}

\subsection{Proofs for Section~\ref{sec4}}
\begin{proof}[Theorem~\ref{wvagouchar}]
    \textit{(i).} The stationary distribution is $\bfY\eqd \bfW(1)$, where $\bfW\sim WVAG^n(a,\bfalpha,\bfnull,\allowbreak\Sigma,\bfeta)$, so by \eqref{v-alpha-g-cf}, it has characteristic exponent
    \begin{align*}
        \Psi_{\bfY}(\bftheta) = \rmi \skal{\bfeta}{\bftheta}-a\log\left(1 + \frac{1}{2}\|\bftheta\|_{\bfalpha\tr \Sigma}^2 \right) - \sum_{k=1}^n \beta_k \log\left(1+\frac{1}{2}\alpha_k\Sigma_{kk}\theta_k^2 \right),
    \end{align*}
    $\bftheta\in\RR^n$. Let $c_0:= -a/(1 + \frac{1}{2}\|\bftheta\|_{\bfalpha\tr \Sigma}^2   ) $ and  $c_k:= -\beta_k/(1+\frac{1}{2}\alpha_k\Sigma_{kk}\theta_k^2) $, $k=1,\dots,n$, then taking the gradient gives
    \begin{align*}
        \nabla_{\bftheta}\Psi_{\bfY}(\bftheta)= \rmi\bfeta+ \frac{c_0}{2}\nabla_{\bftheta} \|\bftheta\|_{\bfalpha\tr \Sigma}^2 + (c_1\alpha_1\Sigma_{11}\theta_1,\dots ,c_n\alpha_n\Sigma_{nn}\theta_n).
    \end{align*}
    By Lemma \ref{conversionlem} (ii),  and noting $\skal{\nabla_{\bftheta} \|\bftheta\|_{\bfalpha\tr \Sigma}^2}{\bftheta} = 2\|\bftheta\|_{\bfalpha\tr \Sigma}^2$, the BDLP $\bfZ$ has characteristic exponent
    \begin{align*}
        \Psi_{\bfZ}(\bftheta) = \skal{\nabla_{\bftheta}\Psi_{\bfY}(\bftheta)}{\bftheta} = \rmi \skal{\bfeta}{\bftheta}  + c_0 \|\bftheta\|_{\bfalpha\tr \Sigma}^2 +\sum_{k=1}^n c_k\alpha_k\Sigma_{kk}\theta_k^2,
    \end{align*}
    which matches \eqref{wvagouz}, as required. Also, $\Psi_{\bfZ}(\bftheta)\to\bfnull$ as $\bftheta\to\bfnull$, as required.
    
    \textit{(ii).} Let $\bfZ_0\sim CP^n (2a,\PPP_{0})$, $\bfZ_k \sim CP^n(2\beta_k,\delta_0^{\otimes (k-1)}\otimes\PPP_{k}\otimes\delta_0^{\otimes (n-k)})$, $k=1,\dots,n$, be independent. Note that $\delta_0^{\otimes (k-1)}\otimes\PPP_{k}\otimes\delta_0^{\otimes (n-k)}$ is the probability law of $VG^n(1,\bfnull,\wt\Sigma_k)$, where the $(k,k)$ element of $\wt\Sigma_k\in\RR^{n\times n}$ is $\Sigma_{kk}$ and all other elements are 0. Then by Lemma~\ref{cpvglem},
    \begin{align}\label{decompz}
        \bfeta I + \sum_{k=0}^n \bfZ_k \eqd \bfZ
    \end{align}
    as the LHS has the same characteristic exponent as \eqref{wvagouz}.
    
    Now if $\bfC_k\sim CP^n (a_k,\wt\PPP_{k})$, $k=0,\dots,n$, are independent, their sum is $ \sum_{k=0}^n{\bfC_k}\sim CP^n\allowbreak (\sum_{k=0}^n a_k,\sum_{k=0}^n\frac{a_k}{\sum_{k=0}^n a_k} \wt\PPP_{k})$. Using this fact, $\bfZ$ is the compound Poisson process with drift in the statement of the result.
    
    \textit{(iii).} From \eqref{cez} and making the substitution $s = e^t$,
    \begin{align*}
        \Psi_{\bfZ^*(\Delta)}(\bftheta) = \int_{1}^{e^{\lambda\Delta}} \frac 1s \Psi_{\bfZ}(s\bftheta)\,\rmd s,\quad \bftheta\in\RR^n.
    \end{align*}
    Then using \eqref{wvagouz},
    \begin{align*}
        \Psi_{\bfZ^*(\Delta)}(\bftheta) = {}&\rmi \skal{\bfeta}{\bftheta}(e^{\lambda\Delta}-1) -a\|\bftheta\|_{\bfalpha\tr \Sigma}^2 \int_{1}^{e^{\lambda\Delta}} \frac{s}{1+\frac{1}{2}\|\bftheta\|_{\bfalpha\tr \Sigma}^2s^2}\,\rmd s \\
        {}& - \sum_{k=1}^n \beta_k\alpha_k\Sigma_{kk}\theta_k^2 \int_{1}^{e^{\lambda\Delta}} \frac{s}{1+\frac{1}{2}\alpha_k\Sigma_{kk}\theta_k^2s^2}\, \rmd s.
    \end{align*}
    Evaluating the above integrals using 
    \begin{align*}
        \int \frac{s}{1+cs^2} \, \rmd s= \frac{\log(1+cs^2)}{2c}+C,
    \end{align*}
    where $c>0$, gives the result in \eqref{wvagouzdelta}.
    
    \textit{(iv).} Note that $\bfZ$ is equal in law to a sum of independent compound Poisson processes and a drift as specified in \eqref{decompz}. Then applying Lemma \ref{cppzlem} gives the result. 
\end{proof}

\begin{proof}[Corollary~\ref{wvagoudenlem}]
    %
    Recall the representation of $\bfZ^*(\Delta)$ given in Theorem~\ref{wvagouchar} (iv) and the notation defined there. Then the probability law of $\bfZ^*(\Delta)$ can be written as $\PPP_{\bfZ^*(\Delta)}=p\PPP + p_1\PPP_1+p_2\PPP_2+p_0\PPP_0$, where $\PPP, \PPP_1,\PPP_2,\PPP_0$ are the conditional probability distributions of
    \begin{align}
        \bfzeta = {}&\bfZ^*(\Delta) \given \{N_0(\lambda\Delta)=0, \,N_1(\lambda\Delta)=0,\, N_2(\lambda\Delta)=0\}, \label{zstarcase1}\\
        \bfzeta+(\wt Z_1,0):={}& \bfZ^*(\Delta) \given \{ N_0(\lambda\Delta)=0,\, N_1(\lambda\Delta)>0,\, N_2(\lambda\Delta)=0\}, \label{zstarcase2}\\
        \bfzeta+( 0,\wt Z_2):={}& \bfZ^*(\Delta) \given \{ N_0(\lambda\Delta)=0,\, N_1(\lambda\Delta)=0, \,N_2(\lambda\Delta)>0\},\label{zstarcase3}\\
        \bfzeta+\wt\bfZ_0:={}& \bfZ^*(\Delta) \given \{ \text{$N_0(\lambda\Delta)=0$, $N_1(\lambda\Delta)>0$, $N_2(\lambda\Delta)>0$; or $N_0(\lambda\Delta)>0$}\},\label{zstarcase4}
    \end{align}
    respectively, and $p,p_1,p_2,p_0$ are the probabilities of the respective events being conditioned on.

    Put $U_1:=\{\bfzeta\}$, $U_2:=S_1$, $U_3 := S_2$, $U_4:=S_0$. Then the measures $\nu_1:=\bfdelta_{\bfzeta}$, $\nu_2:=\LLL\otimes \bfdelta_{\zeta_2}$, $\nu_3:=\bfdelta_{\zeta_1}\otimes \LLL$, $\nu_4:=\LLL^2$ on $\RR^2$ are mutually singular with $\nu_k(U_k^C)=0$ and $\nu_l(U_k)=0$ for all $k\neq l$ by Tonelli's theorem. Therefore, by \cite[Theorem 1]{GoRa08}, the density of $\PPP_{\bfZ^*(\Delta)}$ with respect to $\sum_{k=1}^4\nu_k$ is
    \begin{align*}
        f(\bfz) ={}& p \frac{\rmd \PPP}{\rmd \bfdelta_{\bfzeta}}(\bfz) \eins_{\{\bfzeta\}}(\bfz) + p_1 \frac{\rmd \PPP_1}{\rmd(  \LLL\otimes \bfdelta_{\zeta_2})}(\bfz)  \eins_{S_1}(\bfz) +p_2 \frac{\rmd \PPP_2}{\rmd(  \bfdelta_{\zeta_1}\otimes \LLL)}(\bfz)  \eins_{S_2}(\bfz)\\
        {}& +p_0 \frac{\rmd \PPP_0}{\rmd \LLL^2}(\bfz)   \eins_{S_0}(\bfz),\quad \bfz\in\RR^2.
    \end{align*}
    
    Now we compute the four Radon-Nikodym derivatives. Firstly, ${\rmd \PPP}/{\rmd \bfdelta_{\bfzeta}}(\bfzeta)\allowbreak = 1$. 
    
    
    Let $\Phi_1$ be the characteristic function of $\wt Z_1$, then the characteristic function of  $\sum_{k=1}^{N_1(\lambda\Delta)}\allowbreak e^{T_{1k}} V_{1k}$ can be written as
    \begin{align*}
        \bigg(\frac{1+ \frac 12\alpha_1\Sigma_{11}\theta_1^2e^{2\lambda\Delta}}{1+ \frac 12\alpha_1\Sigma_{11}\theta_1^2} \bigg)^{-\beta_1} ={}& \EE\left[e^{\rmi\bftheta_1{\sum_{k=1}^{N_1(\lambda\Delta)}e^{T_{1k}} V_{1k}}}\givenm N_1(\lambda\Delta)=0\right]\PP( N_1(\lambda\Delta)=0)\\ {}&+\EE\left[e^{\rmi\bftheta_1{\sum_{k=1}^{N_1(\lambda\Delta)}e^{T_{1k}} V_{1k}}}\givenm N_1(\lambda\Delta)>0\right]\PP( N_1(\lambda\Delta)>0)\\
        ={}&  e^{-2\beta_1\lambda\Delta}+ (1-e^{-2\beta_1\lambda\Delta})\Phi_1(\theta_1),\quad\theta_1\in\RR,
    \end{align*} 
    where the LHS comes from similar arguments as in the Proof of Theorem~\ref{wvagouchar} (iii)--(iv). Since $\alpha_1\Sigma_{11}>0$, \cite[Equation (2.10)]{BKMS16} implies that  the distribution of $V_{1k}\sim VG^1(1,0,\alpha_1\Sigma_{11})$ has a Lebesgue density, and consequently so does the distribution corresponding to $ \Phi_1$, which we denote by $f_1$. Hence, for Borel sets $A\subseteq\RR^2$,
    \begin{align*}
        \PPP_1(A) ={}& \PP(\bfZ^*(\Delta)\in A \given N_0(\lambda\Delta)=0, \,N_1(\lambda\Delta)>0, \,N_2(\lambda\Delta)=0) \\
        ={} & \int_{\RR^2}\eins_{A}(\bfz)f_1(z_1-\zeta_1) \,  \rmd z_1 \bfdelta_{\zeta_2}(\rmd z_2),
    \end{align*}
    Therefore, $ {\rmd \PPP_1}/{\rmd(  \LLL\otimes\bfdelta_{\zeta_2})}(\bfz) =f_1(z_1-\zeta_1)$. Similarly,  $ {\rmd \PPP_2}/{\rmd( \bfdelta_{\zeta_1}\otimes\LLL )}(\bfz) =f_2(z_2-\zeta_2)$.

    Finally, let $\Phi_0$ be the characteristic function of $\wt\bfZ_0$, then the characteristic function of $\bfZ^*(\Delta)-\bfzeta$ can be written as
    \begin{align*}
        e^{-\rmi\skal{\bftheta}{\bfzeta}}\Phi_{\bfZ^*(\Delta)}(\bftheta) = p  + p_1 \Phi_1(\theta_1)+ p_2 \Phi_2(\theta_2) + (1-p-p_1-p_2)\Phi_0(\bftheta), \quad \bftheta\in\RR^n, 
    \end{align*} 
    by noting \eqref{zstarcprep} and considering \eqref{zstarcase1}--\eqref{zstarcase4}. Now $|\bfalpha\tr \Sigma|\geq \alpha_1\alpha_2|\Sigma|>0 $ by Oppenheim's inequality and the invertibility of $\Sigma$, so \cite[Equation (2.10)]{BKMS16} implies $\bfV_{0k}\sim VG^2(1,\bfnull,\allowbreak\bfalpha\tr\Sigma)$ has a Lebesgue density. Combining this with the aforementioned fact that $V_{1k},V_{2k}$ also have Lebesgue densities, it follows that $ \Phi_0$ has a Lebesgue density, which we denote by $f_0$. Hence, for Borel sets $A\subseteq\RR^2$,
    \begin{align*}
        \PPP_0(A) ={}& \PP(\bfZ^*(\Delta)\in A \given \text{$N_0(\lambda\Delta)=0$, $N_1(\lambda\Delta)>0$, $N_2(\lambda\Delta)>0$; or $N_0(\lambda\Delta)>0$}
        ) \\
        ={} & \int_{\RR^2}\eins_{A}(\bfz)f_0(\bfz-\bfzeta) \,  \rmd\bfz.
    \end{align*}
    Therefore, $ {\rmd \PPP_0}/{\rmd\LLL^2}(\bfz) =f_0(\bfz-\bfzeta)$, which completes the proof.  
\end{proof}

\begin{proof}[Corollary~\ref{likwvagou}]
    This follows from Lemma \ref{markovllf} as $\bfX$ is in the appropriate form by Remark \ref{indepremark}. We check the parts of the lemma. The existence of the Lebesgue density $f_{\bfX(0)}$ of the stationary distribution $WVAG^2(a,\bfalpha,\bfnull,\Sigma,\bfeta)$ is due to \eqref{wvagpropb}, \cite[Equation (2.10)]{BKMS16} and the assumption that $\Sigma$ is invertible. For $k=1,\dots,m$,
    \begin{align}\label{mtrans}
        M_{\bfx_{k-1}}:\RR^n\to \RR^n,\quad\bfz\mapsto \bfx:= e^{-\lambda\Delta}(\bfx_{k-1}+\bfz)
    \end{align}
    has inverse $M_{\bfx_{k-1}}^{-1}(\bfx) = e^{\lambda\Delta}\bfx-\bfx_{k-1}$, and we take $f_k:= f_{\bfZ^*(\Delta)}$ given in \eqref{zstarden}, and $\nu_k := \nu$ given in \eqref{dommeas}. Also, take $g_{\bfvtheta}\equiv1$, then the dominating measure $\nu^*$ does not depend on the parameter vector $\bfvtheta$ but may depend on the known parameters $\lambda$ and $\bfeta$.  Thus, \eqref{genlik} becomes \eqref{wvagoulik}. 
\end{proof}

\begin{proof}[Proposition~\ref{ouwvagprop}]
    \textit{(i).} The L\'evy measure of $\bfZ$ is given in \cite[Equation (2.20)]{BLM17a}. Let $\bfB^{(\bfalpha)}\sim BM^n(\bfalpha\tr\bfmu,\bfalpha\tr\Sigma)$, $B_k\sim BM^1(\mu_k,\Sigma_{kk})$, $k=1,\dots,n$, and $\GGG_{a,b}$ be the L\'evy measure of the gamma subordinator $\Gamma_S(a,b)$. Consequently, the LHS of the log moment condition \eqref{logcond} becomes
    \begin{align*}
        I_{\log}:={}&\int_{(0,\infty)} \EE[f(\|\bfB^{(\bfalpha)}(g)\|)]\,\GGG_{a,1}(\rmd g) \\
        &+ \sum_{k=1}^n \int_{(0,\infty)}\EE[f(\|B_k(g)\bfe_k\|)]\,\GGG_{\beta_k,1/\alpha_k}(\rmd g)\\
        \le{}& \int_{(0,\infty)} \EE[\|\bfB^{(\bfalpha)}(g)\|^2]\,\GGG_{a,1}(\rmd g) + \sum_{k=1}^n \int_{(0,\infty)}\EE[B_k(g)^2]\,\GGG_{\beta_k,1/\alpha_k}(\rmd g),
    \end{align*}
    where we have used $f(x):= \eins_{(2,\infty)}(x)\log(x)\le x^2$, $x>0$.
    
    The expectation can be evaluated, for example using \cite[Corollory 3.2b.1]{mp92}, giving $\EE[\|\bfB^{(\bfalpha)}(g)\|^2]= \|\bfalpha\tr\bfmu\|^2g^2 + \spur(\bfalpha\tr\Sigma)g$, so
    \begin{align}
        \int_{(0,\infty)}\EE[\|\bfB^{(\bfalpha)}(g)\|^2]\GGG_{a,1}(\rmd g) = \int_{0}^\infty  (a+bg)e^{-g}\rmd g <\infty.
    \end{align}
    for some constants $a,b\geq0$.
    
    Likewise, $\int_{(0,\infty)}\EE[B_k(g)^2]\,\GGG_{\beta_k,1/\bfalpha_k}(\rmd g)<\infty$, $k=1,\dots,n$. Hence, $I_{\log}<\infty$, as required.
    
    \textit{(ii).} Without loss of generality, assume $\bfeta=\bfnull$. Let $\bfV_{0}$, $V_k$, $k=1,\dots,n$, be the independent VG processes defined in \eqref{wvagpropb}. As $\bfZ$ decomposes into a sum of these processes, from \eqref{cez}, we have
    \begin{align*}
        \Psi_{\bfZ^*(\Delta)}(\bftheta) =\int_{0}^{\lambda\Delta} \Psi_{\bfV_0}(e^t\bftheta)\,\rmd t  + \sum_{k=1}^n\int_{0}^{\lambda\Delta} \Psi_{V_k}(e^t\theta_k)\,\rmd t,\quad \bftheta\in\RR^n,
    \end{align*}
    So letting $\bfV_{0}^*(\Delta)$, $V_k^*(\Delta)$ be the independent random variables with characteristic exponents $\bftheta\mapsto\int_{0}^{\lambda\Delta} \Psi_{\bfV_0}(e^t\bftheta)\,\rmd t$, $\theta\mapsto\int_{0}^{\lambda\Delta} \Psi_{V_k}(e^t\theta)\,\rmd t$, respectively, gives $\bfZ^*(\Delta)\eqd\bfV_{0}^*(\Delta)  + \sum_{k=1}^n V_k^*(\Delta)\bfe_k$.
    
    For each $k=1,\dots,n$, $V_k \eqd G_+ - G_-$ for some independent gamma subordinators $G_+, G_-$ (see \cite[Equation (8)]{MCC98}). Like the above argument, we can further decompose $V_k^*(\Delta)\eqd G_+^*(\Delta) -   G_-^*(\Delta)$, where $G_+^*(\Delta), G_-^*(\Delta)$ are the random variables  with characteristic exponents $\theta\mapsto\int_{0}^{\lambda\Delta} \Psi_{G_+}(e^t\theta)\,\rmd t$, $\theta\mapsto\int_{0}^{\lambda\Delta} \Psi_{G_-}(e^t\theta)\,\rmd t$, respectively. By \cite[Equation (4.10)]{QDZ19}, $G_+^*(\Delta)$ and $G_-^*(\Delta)$ are absolutely continuous. This implies, in succession, that $V_k^*(\Delta)$, $\sum_{k=1}^n V_k^*(\Delta)\bfe_k$, and $\bfZ^*(\Delta)$ are absolutely continuous.
    
\end{proof}

\begin{remark}\label{vgdecomp}
    We can be more explicit in the above proof. For $V_k\sim VG^1(1/\alpha_k, \mu_k,\Sigma_{kk})$, we have $G_+\sim \Gamma_S(1/\alpha_k,b_+)$, $G_-\sim \Gamma_S(1/\alpha_k,b_-)$, where $b_{\pm} := (2/\alpha_k)/((\mu_k^2 + 2\Sigma_{kk}/\alpha_k)^{1/2}\pm\mu_k)$. For $G\sim \Gamma_S(a,b)$, the random variable $G^*(\Delta)$ with characteristic exponent $\theta\mapsto\int_{0}^{\lambda\Delta} \allowbreak\Psi_{G}(e^t\theta)\,\rmd t$ satisfies $G^*(\Delta) \eqd e^{\lambda\Delta}(\wt G + \wt C)$, where $\wt G\sim \Gamma(a\lambda\Delta, be^{\lambda\Delta})$, $\wt C\sim  CP^1(a\lambda^2\Delta^2/2,\allowbreak \PPP_J)$, $\PPP_J$ is the probability law such that $J\given U \sim \operatorname{Exp}(be^{\lambda\Delta U})$, $U\sim \operatorname{Uni}(0,1)$. Of course, seeing the gamma random variable $\wt G$ is enough to conclude that $G^*(\Delta)$ is absolutely continuous.\qed
\end{remark}

\subsection{Proofs for Section \ref{sec3}} \label{subsecpf1}

\begin{proof}[Proposition \ref{likprop}] \label{proofcheck}
    This follows from Lemma \ref{markovllf} as $\bfX$ is in the appropriate form by Remark \ref{indepremark}. We check the parts of the lemma. The nondegenerate stationary distribution $\bfX(0)\sim SD^n$ is necessarily absolutely continuous (see \cite[Theorem 27.13]{sato99}). For $k=1,\dots,m$, we take $M_{\bfx_{k-1}}$ given in \eqref{mtrans}, $f_k := f_{\bfZ^*(\Delta)}$ is the Lebesgue density of $\bfZ^*(\Delta)$ which exists by assumption, and $\nu_k:=\LLL$. The transformation theorem for the Lebesgue measure (see \cite[Theorem 19.4]{Ba01}) implies $(\nu_k\circ M_{\bfx_{k-1}}^{-1})(\rmd \bfx_k)= |J(\bfx_k)|\rmd\bfx_k$, where $J(\bfx_k)= e^{n\lambda\Delta}$ is the Jacobian determinant of $M_{\bfx_{k-1}}^{-1}$ at $\bfx_k$, so we can take $g_{\bfvtheta}(\bfx_0,\dots,\bfx_m) :=  e^{mn\lambda\Delta}$ and $\nu^* := \LLL^{m+1}$. Thus, \eqref{genlik} becomes \eqref{likfn}.
    
\end{proof}

\begin{proof}[Proposition~\ref{zconvglem1}]
    {\it (i).}  
    Recall the simple function  $H_{\wt\Delta}$ in \eqref{simplefn} with $\wt\Delta>0$ and $\wt n = \lfloor  \lambda\Delta/\wt\Delta  \rfloor$.  Then for all $s\in[0,\lambda\Delta]$, $H_{\wt\Delta}(s)\to e^s$ pointwisely as $\wt\Delta \to 0$. In addition, $H_{\wt\Delta}(s)\leq e^{\lambda\Delta}$ for all $s\in[0,\lambda\Delta]$, so by the $L^2$ dominated convergence theorem,
    \begin{align*}
        \lim_{\wt\Delta \to 0}\int_0^{\lambda\Delta}|H_{\wt\Delta}(s)-e^s|^2\,\rmd s=0.
    \end{align*}
    Thus, by the definition of the stochastic integral (see \cite[Theorem 2.1]{saya}), we have
    \begin{align}\label{convgz}
        \int_0^{\lambda\Delta} H_{\wt\Delta}(s)\,\rmd\bfZ(s) \stp \int_0^{\lambda\Delta}e^s\,\rmd\bfZ(s)=\bfZ^*(\Delta)
    \end{align}
    as $\wt\Delta \to 0$.
    
    On the other hand, by the definition of the stochastic integral for simple functions (see \cite[Equation (2.4)]{saya}), we have
    \begin{align*}
        \bfZ^*_{\wt\Delta}(\Delta) = \int_0^{\lambda\Delta} H_{\wt\Delta}(s)\,\rmd\bfZ(s) - \bfZ(\wt\Delta) \stp \bfZ^*(\Delta)
    \end{align*}
    as $\wt\Delta \to 0$, where the convergence follows from \eqref{convgz}, $\bfZ(\wt\Delta)\stp 0$ by continuity in probability, and combining these facts using the continuous mapping theorem.
    
    {\it (ii).} Since $(\bfZ^*_{\wt\Delta}(\Delta)^{(k)})_{k=1,\dots,m}$ and $(\bfZ^*(\Delta)^{(k)})_{k=1,\dots,m}$ are iid copies of $\bfZ^*_{\wt\Delta}(\Delta)$ and $\bfZ^*(\Delta)$, respectively, applying Part (i) and the L\'evy continuity theorem yields
    \begin{align}\label{vconvg}
        \hspace{-.01em}\left(\bfX(0),\bfZ^*_{\wt\Delta}(\Delta)^{(1)},\dots,\bfZ^*_{\wt\Delta}(\Delta)^{(m)}\right) \std{}& \left(\bfX(0),\bfZ^*(\Delta)^{(1)},\dots,\bfZ^*(\Delta)^{(m)}\right)
    \end{align}
    as $\wt\Delta \to 0$. Let $g$ be the continuous function such that $g(\bfX(0),\bfZ^*(\Delta)^{(1)},\dots,\allowbreak\bfZ^*(\Delta)^{(m)})=(\bfX(0), \bfX(t_1),\dots ,\bfX(t_m))$ in accordance with \eqref{obsx}. Then applying the continuous mapping theorem to \eqref{vconvg} with the function $g$ completes the proof. 
\end{proof}


\begin{proof}[Lemma \ref{momprop}]
    \textit{(i).} These moment formulas are obtained by differentiating the characteristic function of $\bfZ^*(\Delta)$ given by \eqref{cez}. We show the calculation of \eqref{mom1}.

    Without loss of generality we assume $n=1$, so $Z_1\sim L^1(\mu_1,\Sigma_{11},\ZZZ_1)$. For all $\theta_1\in\RR$, $|\rmi x(e^{\rmi \theta_1 x}-1)|\leq |x|^2$ for $x\in\DD$ and $|\rmi xe^{\rmi \theta_1 x}| \leq |x|$ for $x\in\DD^C$. Now, noting that $\int_{\DD}|x|^2\,\ZZZ_1(\rmd x)<\infty$ as $\ZZZ_1$ is a L\'evy measure and $\int_{\DD^C}|x|\,\ZZZ_1(\rmd x)<\infty$ as $Z_1$ is assumed to have a finite first moment (see \cite[Example 25.12]{sato99}), differentiating \eqref{lkform} using the dominated convergence theorem gives
    \begin{align*}
        \Psi_{Z_1}'(\theta_1) = \rmi \mu_1 -  \Sigma_{11}\theta_1 +\int_{\DD} \rmi x(e^{\rmi \theta_1 x}-1)\,\ZZZ_1(\rmd x)+ \int_{\DD^C} \rmi xe^{\rmi \theta_1 x}\,\ZZZ_1(\rmd x).
    \end{align*}
    Thus, for $\theta_1$ in a neighbourhood of 0, there exist constants $a,b>0$, such that $t\mapsto |\partial_{\theta_1} \allowbreak(\Psi_{Z_1}(e^t \theta_1))| \leq ae^t +be^{2t}$ on $t\in[0,\lambda\Delta]$, so that \eqref{cez} can also be differentiated using the dominated convergence theorem, giving
    \begin{align*}
        \partial_{\theta_1} \int_{0}^{\lambda\Delta} \Psi_{Z_1}(e^t \theta_1)\,\rmd t = \int_{0}^{\lambda\Delta} \Psi_{Z_1}'(e^t \theta_1)e^t\,\rmd t.
    \end{align*}
    Thus,
    \begin{align*}
        \EE[Z^*_1(\Delta)] ={}& \frac{1}{\rmi} \partial_{\theta_1} \Phi_{Z^*_1(\Delta)}(\theta_1) \left.\right|_{\theta_1=0} \\
        ={}& \frac{1}{\rmi}  \Phi_{Z^*_1(\Delta)}(0) \int_{0}^{\lambda\Delta} \Psi_{Z_1}'(0)e^t\,\rmd t\\
        ={}&\int_{0}^{\lambda\Delta} \EE[Z_1]e^t\,\rmd t\\
        ={}&(e^{\lambda\Delta}-1)\EE[Z_1],
    \end{align*}
    as required. 
    
    Using a similar method, the higher order moments follow.
    
    \textit{(ii).} This follows immediately from \cite[Propsotion 2.6]{mas04}. 
\end{proof}

\begin{proof} [Corollary~\ref{wvagoumom}]
    By Lemma \ref{momprop} (i), this reduces to finding the corresponding moments of the BLDP $\bfZ=(Z_1,\dots,Z_n)$ given in Theorem~\ref{wvagouchar} (ii) and we recall the notation there. Then by \cite[Equation (25.8)]{sato99}, $\myCov(Z_k,Z_l) = b \EE[J_kJ_l] = 2a\EE[V_kV_l]$, $k\neq l$, where $\bfJ=(J_1,\dots,J_n)$ is the random vector with probability law $\PPP$, and $\bfV=(V_1,\dots,V_n)\sim VG^n(1,\bfnull,\bfalpha\tr\Sigma)$. Next, the moments of $Z_k\sim CP^1(2/\alpha_k,\PPP_{V_k},\allowbreak\eta_k)$ can be written in terms of the moments of $V_k$ (for example, \cite[Section 5]{GT06}). Finally, the moments of $\bfV$ can be determined by combining \cite[Remark 4, Appendix A.1, Figure 2]{MiSz17}.
    
\end{proof}

\appendix
\section{The Connection Between LDOUPs and Self-De\-composability}\label{append}

In this appendix, we review the well-established connection between LDOUPs and self-decomposability.

All self-decomposable distributions are infinitely divisible and there is a one-to-one correspondence between the stationary solutions of LODUPs and self-decomposable distributions, which we summarise in the following lemma (see \cite[Theorems 17.5 and 17.11]{sato99}).

\begin{lemma}\label{statsolnlem} Fix $\lambda >0$ and let $\bfX$ be the LDOUP given by \eqref{ousoln} with BDLP $\bfZ\sim L^n(\bfmu,\Sigma,\ZZZ)$.
    \begin{enumerate}[(i)]
        \item[(i)] For all $\bfZ\sim L^n(\bfmu,\Sigma,\ZZZ)$ satisfying \eqref{logcond}, there exists a $\bfY\sim SD^n$ such that $\bfX$ has stationary distribution $\bfY$.
        \item[(ii)] For all $\bfY\sim SD^n$, there exists a $\bfZ\sim L^n(\bfmu,\Sigma,\ZZZ)$ satisfying \eqref{logcond}, unique in law, such that $\bfX$ has stationary distribution $\bfY$.
        \item[(iii)] If $\bfZ\sim L^n(\bfmu,\Sigma,\ZZZ)$ does not satisfy \eqref{logcond}, then $\bfX$ has no stationary distribution.
    \end{enumerate}
\end{lemma}

Furthermore, it is possible to convert between the characteristic exponents of the stationary distribution $\bfY$ and the BDLP $\bfZ$ using the next result (see \cite[Theorem 17.5]{sato99} for (i) and \cite[Lemma 2.5]{mas04} for (ii)).
\begin{lemma}\label{conversionlem}
    Fix $\lambda >0$ and let $\bfX$ be the LDOUP given by \eqref{ousoln}.
    \begin{enumerate}
        \item[(i)] Let $\bfZ\sim L^n(\bfmu,\Sigma,\ZZZ)$ satisfying \eqref{logcond} be the BDLP of $\bfX$, then the stationary distribution $\bfY$ has characteristic exponent \eqref{cey}.
        
        \item[(ii)] Let $\bfY\sim SD^n$ be the stationary distribution of $\bfX$. Suppose $\Psi_{\bfY}$ is differentiable for all $\bftheta\neq \bfnull$ and $\skal{\nabla_{\bftheta} \Psi_{\bfY} (\bftheta)}{\allowbreak\bftheta}\to\bfnull$ as $\bftheta \to\bfnull$, then the BDLP $\bfZ$ has characteristic exponent
        \begin{align*}
            \Psi_\bfZ(\bftheta) =  \skal{\nabla_{\bftheta} \Psi_{\bfY} (\bftheta)}{\bftheta},\quad \bftheta\in\RR^n.    
        \end{align*}
    \end{enumerate}
\end{lemma}


The next lemma is about $\bfZ^*(\Delta)$ defined in \eqref{zstar}, it is implied by \cite[Theorem 2.2]{saya} and Kac's theorem. It explains why the observations of a LDOUP form a AR(1) process, and why the stationary solution of a LDOUP gives rise to self-decomposable distributions.

\begin{lemma}\label{cor1} Let $\Delta>0$ and $\bfZ\sim L^n$. For $t_0=0,t_1=\Delta , \dots,t_m= m\Delta$,
    \begin{align*}
        \Bigg( {\int_{t_{k-1}}^{t_k} e^{\lambda s}\,\rmd\bfZ(\lambda s)}\Bigg)_{k=1,\dots,m}
    \end{align*}
    is an iid sequence equal in distribution to $\bfZ^*(\Delta)$, which has characteristic exponent \eqref{cez}.
\end{lemma}

\begin{remark}\label{indepremark}
    Applying \eqref{ousoln} at the times $t=t_0,t_1,\dots,t_m$, we have
    \begin{align}\label{discXeqn}
        \bfX(t_{k})= b\bfX(t_{k-1})+\bfZ_{b}^{(k)},\quad k=1,\dots,m,
    \end{align}
    where $b=e^{-\lambda\Delta}$ and $\bfZ_{b}^{(k)}=e^{-\lambda \Delta} \int_{t_{k-1}}^{t_k} e^{\lambda s}\,\rmd\bfZ(\lambda s)$. Now by Lemma \ref{cor1}, $\bfZ_{b}^{(k)}\eqd e^{-\lambda\Delta}\bfZ^*(\Delta)$, $k=1,\dots,m$, are iid, and $\bfX(t_{k-1})$, being a function of $\bfX_0,\bfZ_{b}^{(1)},\dots, \bfZ_{b}^{(k-1)}$ only, is independent of $\bfZ_{b}^{(k)}$. \qed
\end{remark}

From \eqref{discXeqn}, the observations $\bfX(0),\dots,\bfX(t_m)$ follow an AR(1) process with innovation terms $\bfZ_{b}^{(k)}$, $k=1,\dots,m$. With a minor abuse of terminology, we call $\bfZ^*(\Delta)$ the innovation term. As noted in \cite[Sections 3 and 5]{bnjs}, \eqref{discXeqn} satisfies \eqref{sddefn} with stationary distribution $\bfY\eqd \bfX(t_k)\eqd \bfX(t_{k-1})$, $b=e^{-\lambda\Delta}$ and $\bfZ_b=\bfZ_{b}^{(k)}$. This demonstrates the connection between the stationary distribution of a LDOUP and self-decomposability.

\subsection*{Acknowledgments}

This research was partially supported by ARC grant DP160104737.The author thanks Boris Buchmann for discussions and suggestions, and two anonymous referees for their helpful comments and suggestions. The author also thanks Liwei Cao for work on part of the code.


\subsection*{Conflict of interest}
The author declares that they have no conflict of interest.

\subsection*{Data availability statement}
The code used in this paper is available at \url{https://github.com/klu5893/LDOUP-Calibration}.

\bibliographystyle{abbrv}
\bibliography{bibliography}
\addcontentsline{toc}{section}{\refname}

\end{document}